\documentclass[11pt, onesided, reqno]{amsart}

\usepackage{amssymb}
\usepackage{amsmath}
\usepackage{enumerate}
\usepackage{graphicx}
\usepackage{tikz}

\evensidemargin\oddsidemargin

\newtheorem{theorem}{Theorem}
\newtheorem{proposition}[theorem]{Proposition}

\newtheorem{lemma}[theorem]{Lemma}

\theoremstyle{definition}

\newtheorem{remark}[theorem]{Remark}

\newcommand{\eqnsection}{
\renewcommand{\theequation}{\thesection.\arabic{equation}}
    \makeatletter
    \csname  @addtoreset\endcsname{equation}{section}
    \makeatother}
\eqnsection

\def\e{\mathbf{e}}
\def\g{\mathbf{g}}
\def\E{\mathbb{E}}

\def\N{\mathbb{N}}

\def\R{\mathbb{R}}

\def\Pb{\mathbb{P}}

\def\F{\mathcal{F}}

\newcommand{\equi}{\; \mathop{\sim}\limits}
\def\={{\,\;\mathop{=}\limits^{\text{(law)}}\;\,}}

\def\qed{\hfill$\square$}

\allowdisplaybreaks

\makeatletter
\@namedef{subjclassname@2020}{%
  \textup{2020} Mathematics Subject Classification}
\makeatother

\begin{document}

\title[]{Total progeny for spectrally negative branching L\'evy processes with absorption}
\author[Christophe Profeta]{Christophe Profeta}

\address{
Universit\'e Paris-Saclay, CNRS, Univ Evry, Laboratoire de Math\'ematiques et Mod\'elisation d'Evry, 91037, Evry-Courcouronnes, France.
 {\em Email} : {\tt christophe.profeta@univ-evry.fr}
  }



\begin{abstract} 
We consider a spectrally negative branching L\'evy process in which particles are killed upon crossing below zero. It is known that such a process becomes extinct almost surely if the drift toward $-\infty$ is sufficiently strong to counterbalance the reproduction rate. In this note, we study the tail asymptotics of the number of particles absorbed at the boundary during the lifetime of the process, in both the subcritical and critical regimes.
\end{abstract}

\maketitle

\section{Introduction}

\subsection{Definitions and context}
We consider a spectrally negative branching L\'evy process $X$ in the sense of Kyprianou \cite{Kyp1} in which the particles are killed when entering $(-\infty, 0]$. More precisely, the process starts with an initial particle located at $a>0$ and has the following properties:
\begin{enumerate}
\item Each particle lives for an exponentially distributed time of parameter $\beta>0$, after which it dies and splits into a random number $\boldsymbol{p}$ of offspring.
\item Between branching events, each particle moves independently according to a spectrally negative L\'evy process $L$, starting from the location of its parent's death.
\item Particles are killed and removed from the system upon entering $(-\infty,0]$.
\end{enumerate}
We denote by ${\bf Z}_0$ the number of particles that have entered $(-\infty, 0]$ during the lifetime of the process. The aim of this note is to analyze the tail asymptotics of ${\bf Z}_0$, under assumptions ensuring that  $X$ becomes extinct almost surely. \\

In what follows, we shall assume that all the processes and random variables are defined on the same probability space $(\Omega, \F, \Pb)$, and we shall denote  by $\Pb_a$, with an abuse of notation, both the laws of $X$ and $L$ when started from $a$. We denote by $m:=\E[\boldsymbol{p}]$ the expectation of $\boldsymbol{p}$ (which is assumed to be finite) and write $p_n:=\Pb(\boldsymbol{p}=n)$ for $n\geq0$.
 In the case of the branching Brownian motion with negative drift, the exact asymptotic behavior of $\Pb_a({\bf Z}_0 =  n)$ has been obtained by Maillard \cite{Mai} by analyzing the generating function of ${\bf Z}_0$, using differential equations and the fact that Brownian motion has no negative jumps. Berestycki, Brunet,  Harris \& Mi\l{}o\'s \cite{BBHSM} have then studied the case of branching Brownian motion with positive drift: in this case the particles all drift to $+\infty$, but  ${\bf Z}_0$ may remain finite although the branching process $X$ does not die out a.s. These approaches seem hardly generalizable here, and we shall rather follow the ideas of A\"id\'ekon \cite{Aid} and A\"id\'ekon, Hu \& Zindy \cite{AHZ}, who have studied the total progeny of branching random walks. We recall their main result below. \\

Let $\mathcal{V}$ be a branching random walk starting from some $x\in \R$ and  such that the underlying Galton-Watson  tree $\mathcal{T}$ is supercritical. The authors denote by $\psi$ its logarithmic generating function and assume that there exists $\varrho_\ast>0$ such that $\psi(\varrho_\ast) =\varrho_\ast \psi^\prime(\varrho_\ast)$. They further introduce a killing barrier at 0:  any particle which enters $(-\infty,0)$ is immediately removed from the system. Now, it is known that, conditioning on the fact that $\{\sharp\mathcal{T}=\infty\}$, the rightmost particle in the
branching random walk without killing has a negative speed when $\psi^\prime(\varrho_\ast)<0$, and converges almost surely to $-\infty$ in the logarithmic scale when $\psi^\prime(\varrho_\ast)=0$. In both cases, the process thus dies out a.s. and one can define its total progeny $\mathcal{Z}$. The following result was then proven in \cite{AHZ}, under some further conditions on $\psi$ and on the reproduction law of $\mathcal{V}$:
\begin{enumerate}
\item Case  $\psi^\prime(\varrho_\ast)<0$. There exist two positive numbers
$\varrho_- < \varrho_\ast<\varrho_+$ such that $\psi(\varrho_-) = \psi(\varrho_+)=0$ and a constant $c_{\emph{sub}}>0$ such that 
$$\Pb_x(\mathcal{Z}\geq n) \equi_{n\rightarrow+\infty}  c_{\emph{sub}} \,R(x) e^{\rho_+ x} n^{-\varrho_+ / \varrho_-},\qquad n\rightarrow+\infty.$$
\item Case $\psi^\prime(\varrho_\ast)=0$.  There exists a constant $c_{\emph{crit}} >0$ such that 
$$\Pb_x(\mathcal{Z}\geq n) \equi_{n\rightarrow+\infty}  c_{\emph{crit}}\, R(x) e^{\rho_\ast x}\frac{1}{n (\ln(n))^2},   \qquad n\rightarrow+\infty.$$
\end{enumerate}
In both cases, $R(x)$ denotes the renewal function of  the associated random walk under a specific change of measure. We shall see in Theorem \ref{theo:1} how all those terms translate in the context of branching L\'evy processes. In particular, the role of the renewal function $R$ will be played by the scale function of the L\'evy process $L$ under a certain exponential change of measure. This seems quite natural as both terms admit a similar expression in terms of the ladder height processes, see \cite[p. 3801]{AHZ} and \cite[Equation (9.7)]{Kyp}. \\

There have been many works on such branching L\'evy processes lately, and in particular on their global maximum ${\bf M}$, i.e. the maximum location ever reached by one of the particles. For such a maximum to be well-defined, one generally gives conditions ensuring that the branching process dies out a.s. This has been done in the literature through two situations: 
\begin{enumerate}[(i)]
\item First, when $m:=\E[\boldsymbol{p}]$ is smaller or equal to 1 so that the underlying Galton-Watson tree is a.s. finite. Some first results were obtained under a third moment condition on $\boldsymbol{p}$ for stable L\'evy processes \cite{LaSh, Pro1} and then spectrally negative L\'evy processes \cite{Pro2}. These results were then extended to the case where  $\boldsymbol{p}$ is in the domain of attraction of a stable distribution \cite{HJRS, Pro3, HJRS3}.
\item Second, when one adds a killing barrier together with some conditions ensuring that the process is in some sense attracted to the barrier. The asymptotics of the maximum ${\bf M}$ has been studied in \cite{Pro, HJRS2, RSZ} under different sets of conditions on $\boldsymbol{p}$. Our set-up here lies in this second situation.
\end{enumerate}

We finally mention that branching L\'evy processes have been defined in a far more general set-up in \cite{BeMa}. The authors show that one can allow branchings at infinite rate, infinitely many particles to be born at a branching event, and new particles to be born with displacements relatively to the parent. In this general set-up, much emphasis has been made so far on studying the convergence of the associated martingales, see \cite{BeMa2, MaSh, RSZ2}. Closer to our study, we also mention a recent paper on the leftmost particle of branching subordinators \cite{KaVe}.\\

\subsection{Statement of the main result}

For $\lambda \geq0$, the Laplace exponent of $L$ is given by 
\[\Psi(\lambda) :=\ln \E_0\left[e^{\lambda L_1}\right]=d\lambda +  \frac{\sigma^2}{2}\lambda^2 + \int_{-\infty}^0 \left(e^{\lambda x}- 1 - \lambda x 1_{\{|x|<1\}}\right) \nu(dx)\]
where $d \in \R$ is the drift coefficient, $\sigma \geq0$ the Gaussian coefficient, and $\nu$ is the L\'evy measure satisfying the standard integrability condition $ \int_{-\infty}^0 (x^2\wedge 1)\, \nu(dx)<+\infty$. 
We assume that the one-dimensional distributions of $L$ are absolutely continuous, i.e., 
\[\Pb_0(L_t\in dx) \ll dx, \qquad \text{for every }t>0\]
and we exclude the case where $-L$ is a subordinator.  As a consequence, the function $\Psi$ is strictly convex and tends to $+\infty$ as $\lambda\rightarrow +\infty$. 
One necessary condition for the branching process $X$ to become extinct a.s. is that $\Psi^\prime(0^+)<0$.  In this case, by strict convexity, $\Psi$ admits a unique minimum at $\lambda_\ast>0$ which is such that $\Psi^\prime(\lambda_\ast)=0$ and $\Psi^{\prime\prime}(\lambda_\ast) = \sigma^2 + \int_{-\infty}^{0} x^2 e^{\lambda_\ast x} \nu(dx)>0$. Setting $-q_\ast := \Psi(\lambda_\ast)$, this implies that for any $-q\geq -q_\ast$, the equation $\Psi(\lambda)=-q$ admits at most two solutions, and we denote by $\Phi_-(-q)$ the smaller one and by $\Phi_+(-q)$ the larger one:
\[\Phi_-(-q) := \inf\{\lambda\geq 0,\, \Psi(\lambda)=-q\}\qquad \text{and}\qquad \Phi_+(-q) := \sup\{\lambda\geq 0,\, \Psi(\lambda)=-q\}.\]
To simplify the notation, we shall simply write $\Phi(-q_\ast)$ when $q=-q_\ast$ since both solutions are equal in this case.
\begin{center}
\begin{tikzpicture}

\draw[thick, ->,>=stealth] (0,-2) --(0,2.5);
\draw[thick, ->,>=stealth]  (-1,0) --(7,0);
\draw (0,0) ..controls +(2,-3) and +(-1,-3).. (7,2.5);
\draw[dashed] (0,-1.36) -- (2, -1.36);
\draw (-.4, -1.36) node{$-q_\ast$};

\draw[dashed] (0,-0.7) -- (4.4, -0.7);
\draw (-.4, -0.7) node{$-q$};
\draw (-0.2, -0.2) node{$0$};

\draw[dashed] (0.59,-0) -- (0.59, -0.7);
\draw (.75, 0.25) node{$\Phi_-(-q)$};

\draw[dashed] (4.4, 0) -- (4.4, -0.7);
\draw (4.5, 0.25) node{$\Phi_+(-q)$};

\draw[dashed] (2.3,-0) -- (2.3, -1.36);
\draw (2.4, 0.25) node{$\Phi(-q_\ast)$};
\draw (2.4, 0.7) node{\rotatebox{90}{=}};
\draw (2.4, 1.1) node{$\lambda_\ast$};

\draw (6.4, 2) node{$\Psi$};

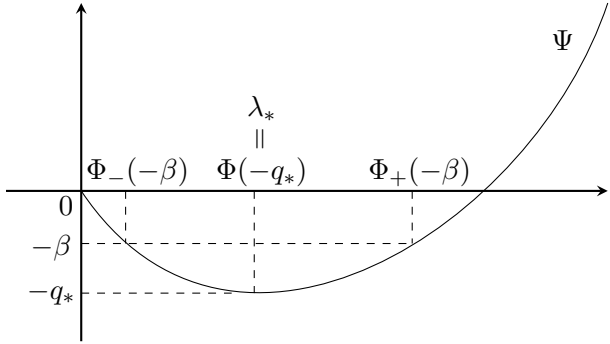
\draw (2, -2.5) node{Graph of the function $\Psi$};
\end{tikzpicture}
\end{center}

We now show that the condition
\begin{equation}\label{eq:extinct}
\Psi^\prime(0^+)<0 \qquad \text{ and }\qquad \beta(m-1)\leq q_\ast
\end{equation}
is sufficient to ensure that the process $X$ dies out a.s. Let us denote by $N_t$ the number of particles alive at time $t$, and $\zeta:=\inf\{t\geq0,\; N_t=0\}$ the extinction time. Since $N_t$ is integer-valued, we have, using the  many-to-one formula, see \cite[Prop. 3.3]{BDMT},
\begin{equation}\label{eq:boundv}
\Pb_a(\zeta> t) =\Pb_a(N_t\geq 1) \leq \E_a[N_t] = \Pb_a(\tau_0^->t) e^{\beta(m-1) t},
\end{equation}
where  $\tau_0^- := \inf\{t\geq 0,\, L_t\leq 0\} $ denotes the first time $L$ goes below 0.
But, from  Kyprianou \& Palmowski \cite[Corollary 4]{KP}, there exists a function $\kappa$ such that
\begin{equation}\label{eq:asympT0}
\Pb_a(\tau_0^- >t) \equi_{t\rightarrow +\infty}  \kappa(a)  t^{-3/2} e^{-q_\ast t}.
\end{equation}
Therefore, passing to the limit in (\ref{eq:boundv}), we thus deduce that $\lim\limits_{t\rightarrow +\infty}\Pb_a(\zeta> t) =0$. As a consequence, $\Pb_a(\zeta=+\infty) = 0$, i.e., the process dies out a.s. 
In other words, while the condition $\Psi^\prime(0^+)<0$ states that each particle must drift toward $-\infty$,  the second condition of (\ref{eq:extinct}) requires that such a drift must be  strong enough to counterbalance the reproduction mechanism. This assumption shall thus be in force throughout the paper.  
In fact, one may prove that condition (\ref{eq:extinct}) is also necessary (when $m> 1$) following the argument of \cite{Pro} which deals with the dyadic case, i.e. $\Pb(\boldsymbol{p}=2)=1$, but this will not be used in this paper. In what follows, we shall say that $X$ is subcritical when $\beta(m-1)<q_\ast$, and critical if $\beta(m-1)=q_\ast$. For such a problem, the critical case is significantly more difficult to handle: one has to write down second order expansions, while first order expansions are enough to deal with the subcritical case. Here, via a change of measure, we will see that the subcritical case requires to work with transient L\'evy processes, while the critical case requires to handle recurrent L\'evy processes.\\

\noindent
Before stating our main theorem, we recall the definition of the scale function $W^{(q)}$ for  $q\geq 0$.  $W^{(q)}$ is an increasing function which is null on $(-\infty,0)$ and is defined   on $(0,+\infty)$ by: 
\begin{equation}\label{eq:W}
\int_0^{+\infty}e^{-\lambda x} W^{(q)}(x)dx = \frac{1}{\Psi(\lambda)-q},\qquad \lambda > \Phi_+(q).
\end{equation}
Moreover, for every $a>0$, the function $q\rightarrow W^{(q)}(a)$ admits an analytic extension to the complex plane, see Kyprianou \cite[Section 8.3]{Kyp}. When $q=0$, we shall write $W:=W^{(0)}$ for simplicity. 

\begin{theorem}\label{theo:1}
Assume that $a\geq 1$ and $\Psi^\prime(0^+)<0$. Let us set $\gamma=\left\lfloor \frac{\Phi_+(\beta(1-m))}{\Phi_-(\beta(1-m))}\right\rfloor$ and assume that $m=\E[\boldsymbol{p}]>1$ and that $\E[\boldsymbol{p}^{\gamma+1}]<+\infty$.
\begin{enumerate}
\item Suppose that $\beta(m-1)<q_\ast$.
Then there exist two constants $\kappa_1, \kappa_2>0$ independent of $a\geq 1$, such that as $n\rightarrow +\infty$,
$$ \kappa_1  W^{(\beta(1-m))}(a) n^{-\frac{\Phi_+(\beta(1-m))}{\Phi_-(\beta(1-m))}} \leq \Pb_a({\bf Z}_0\geq n)\leq \kappa_2  W^{(\beta(1-m))}(a) n^{-\frac{\Phi_+(\beta(1-m))}{\Phi_-(\beta(1-m))}}.$$
\item Suppose that  $\beta(m-1)=q_\ast$ and that the function $x\rightarrow e^{-\Phi(-q_\ast) x} W^{(-q_\ast)}(x)$ is concave. 
Then there exist two constants $\kappa_1, \kappa_2>0$ independent of $a\geq 1$, such that as $n\rightarrow +\infty$,
$$\kappa_1W^{(-q_\ast)}(a) \frac{1}{n \ln^2(n)}  \leq \Pb_a({\bf Z}_0\geq n)\leq \kappa_2 W^{(-q_\ast)}(a)\frac{1}{n \ln^2(n)}.$$
\end{enumerate}
\end{theorem}

\begin{remark}
The lower bounds are in fact valid with $ \kappa_1$ independent of $a$ for any $a>0$. For the upper bounds, one can replace $a\geq 1$ by $a\geq \delta$ for any $\delta>0$. We believe these estimates to be valid with $\kappa_2$  independent of $a$ for any $a>0$, but our strategy of proof requires $a$ to be separated from 0 to get the upper bound.
\end{remark}

\begin{remark}\label{rem:concave}
The concavity assumption is essentially technical, and will allow us to use Karamata's Tauberian theorem  \cite[Theorem 1.7.6]{BGT}, as well as to control the sign of certain quantities.
Let us define, for $c\geq \Phi(-q_\ast)$, the Esscher transform:
\begin{equation}\label{eq:Esscher}
\frac{\text{d}\Pb_a^{c}}{\text{d}\Pb_a}\Bigg|_{\F_t} := e^{c(L_t-a) -\Psi(c) t},
\end{equation}
where $(\F_t)$ denotes the natural filtration of $L$. Under $\Pb_a^{c}$, the process $L$ remains a spectrally negative L\'evy process, with modified characteristics denoted by:
\begin{equation}\label{eq:esscherparam}
\begin{array}{ccc}
\Psi_c(\lambda):= \Psi(\lambda +c) - \Psi( c), &\qquad  \nu_c(dx) := e^{cx }\nu(dx),  &\qquad  W_c(x) := e^{-cx} W^{(\Psi(c))}(x).\\
\end{array}
\end{equation}
By taking $c=\Phi(-q_\ast)$,  the concavity assumption in Theorem \ref{theo:1} is thus equivalent to the fact that the function
$W_{\Phi(-q_\ast)}$ is concave for $z\geq0$. From Kyprianou, Rivero \& Song \cite[Theorem 2.1]{KRS}, this holds, for instance, when $\nu_{\Phi(-q_\ast)}$ is absolutely continuous with respect to the Lebesgue measure and such that $x\longmapsto \nu_{\Phi(-q_\ast)}((-\infty, -x))$ is log-convex.  A classic example is given by spectrally negative $\alpha$-stable 
L\'evy processes with negative drift. Their  Laplace exponents and L\'evy measures are given by 
$$\Psi(\lambda) = d \lambda + \eta \lambda^\alpha   \qquad \text{ and }\qquad \nu(dx) = \frac{c }{|x|^{1+\alpha}}1_{\{x<0\}} dx$$
where the self-similarity parameter $\alpha\in(1,2)$, the drift coefficient $d=\Psi^\prime(0^+) <0$ and the constants $c,\eta>0$ are normalization factors. Here $q_\ast = 
(\alpha-1)\eta (-\frac{d}{\eta\alpha})^{\alpha/(\alpha-1)}$ and $\Phi(-q_\ast)  =\lambda_\ast=  (-\frac{d}{\eta\alpha})^{1/(\alpha-1)}$. Now, from (\ref{eq:esscherparam}), the log-convexity condition amounts to prove that the function
$$f(x)= \int_x^{+\infty} \frac{e^{-\lambda_\ast z}}{z^{1+\alpha}} dz$$
is log-convex. Using the integral representation of the Gamma function and Fubini-Tonelli theorem, we write
\begin{align*}
 f(x)=\int_x^{+\infty} \frac{e^{-\lambda_\ast z}}{z^{1+\alpha}} dz &=   \frac{1}{\Gamma(1+\alpha)} \int_x^{+\infty} e^{-\lambda_\ast z}  \int_0^{+\infty} y^{\alpha} e^{-yz }dydz\\
 &= \frac{1}{\Gamma(1+\alpha)} \int_0^{+\infty}   y^{\alpha}    \int_x^{+\infty} e^{-(\lambda_\ast+y) z} dz dy\\
 &=\frac{1}{\Gamma(1+\alpha)} \int_0^{+\infty}  \frac{ y^{\alpha} }{\lambda_\ast+y}  e^{-(\lambda_\ast+y) x} dy.
\end{align*}
This last expression shows that $f$ is completely monotone, hence it is log-convex.

\end{remark}

Our proof of Theorem \ref{theo:1} is essentially inspired by A\"id\'ekon \cite{Aid}. The idea is that for ${\bf Z}_0$ to be very large, some particles must move far away from the boundary in order to have a large progeny. Therefore, the asymptotic behavior of ${\bf Z}_0$ is, to some extent, linked to that of the overall maximum ${\bf M}$ of the branching process. The exact asymptotics of ${\bf M}$ has been computed in the case of dyadic branching in \cite{Pro}. In our set-up, we have the following estimates:

\begin{proposition}\label{prop:M} 
Assume that $m=\E[\boldsymbol{p}]>1$ and $\E[\boldsymbol{p}^2]<+\infty$.
\begin{enumerate}
\item If $\beta(m-1)<q_\ast$,  there exist two constants $\kappa_1, \kappa_2>0$ independent of $a>0$ such that as $x\rightarrow +\infty$:
\begin{equation}\label{eq:Mbeta}
  \kappa_1 W^{(\beta(1-m))}(a)  e^{-\Phi_+(\beta(1-m)) x}\leq \Pb_a\left({\bf M} \geq x\right)  \leq \kappa_2 W^{(\beta(1-m))}(a)  e^{-\Phi_+(\beta(1-m)) x}.
   \end{equation}
\item If $\beta(m-1)=q_\ast$,  there exist two constants $\kappa_1, \kappa_2>0$ independent of $a>0$ such that as $x\rightarrow +\infty$:
\begin{equation}\label{eq:AsymptMq}
 \kappa_1 W^{(-q_\ast)}(a)  \frac{e^{-\Phi(-q_\ast) x}}{x}\leq  \Pb_a\left({\bf M} \geq x\right) \leq \kappa_2 W^{(-q_\ast)}(a)  \frac{e^{-\Phi(-q_\ast) x}}{x}.
 \end{equation}
\end{enumerate}
\end{proposition}
In particular,  the appearance of the $1/x$ factor in the critical case explains the presence of logarithmic terms in the asymptotics of ${\bf Z}_0$.

\begin{remark}
The Esscher transform implies for $\beta(m-1)\leq q_\ast$ the equalities:
\begin{equation}\label{eq:W+=W-}
 W^{(\beta(1-m))}(a) = e^{\Phi_-(\beta(1-m))a} W_{\Phi_-(\beta(1-m))}(a) = e^{\Phi_+(\beta(1-m))a} W_{\Phi_+(\beta(1-m))}(a).
\end{equation}
This allows to compute the asymptotics of all these quantities. Indeed, on the one hand, when $\beta(m-1)<q_\ast$, then $\Psi^\prime(\Phi_+(\beta(1-m)))>0$ and from  \cite[Lemma 3.3]{KKR}:
\begin{equation}\label{eq:limWbeta}
W_{\Phi_+(\beta(1-m))}(a) \xrightarrow[a\rightarrow+\infty]{}\frac{1}{\Psi^\prime_{\Phi_+(\beta(1-m))}(0^+)} =\frac{1}{\Psi^\prime(\Phi_+(\beta(1-m)))}.
\end{equation}
On the other hand, when  $\beta(m-1)=q_\ast$, we have $\Psi^\prime(\Phi(-q_\ast))=0$ and by the definition (\ref{eq:W}) of $W_{\Phi(-q_\ast)}$, it holds
$$\int_0^{+\infty}e^{-\lambda x} W_{\Phi(-q_\ast)}(x)dx = \frac{1}{\Psi_{\Phi(-q_\ast)}(\lambda)} = \frac{1}{\Psi(\lambda +\Phi(-q_\ast)) +q_\ast} \equi_{\lambda\downarrow0} \frac{2}{\Psi^{\prime\prime}(\Phi(-q_\ast)) \lambda^2}.$$
Since $W_{\Phi(-q_\ast)}$ is increasing, we deduce from Karamata's Tauberian theorem that
\begin{equation}\label{eq:limWq}
W_{\Phi(-q_\ast)}(a) \equi_{a\rightarrow+\infty} \frac{2}{\Psi^{\prime\prime}(\Phi(-q_\ast))} a.
\end{equation}
\end{remark}

The rest of the paper is organized as follows: in Section \ref{sec:2}, we write down a recursive equation which is at the core of our proof as it will allow us to control the moments of some quantities related to ${\bf Z}_0$. We then give in Section \ref{sec:3} the proof of Proposition \ref{prop:M}. The proof of Theorem \ref{theo:1} in the subcritical case is given in Section \ref{sec:4}, and the critical case is handled in Section \ref{sec:5}.

\section{Preliminary computations}\label{sec:2}

We start by collecting some preliminary results that apply both to the subcritical case $\beta(m-1)<q_\ast$ and to the critical case $\beta(m-1)=q_\ast$. 

\subsection{A recurrence equation}

To study the random variable ${\bf Z}_0$, we shall introduce an upper barrier at some fixed level $x>a$, and then let $x$ go to $+\infty$ along some well-chosen sequence. We denote by ${\bf Z}_{0<x}$ the total number of individuals that enter $(-\infty,0]$ during the lifetime of the process, and whose ancestors have always stayed in $(0,x)$. Let us define the following passage times:
$$\tau_0^- = \inf\{t\geq0, \; L_t\leq 0\} \qquad \text{ and }\qquad \tau_x^+ = \inf\{t\geq0, \; L_t\geq x\}.$$

\begin{lemma}\label{lem:Mn}
Let $a>0$, $n\in \N$ and suppose that  $\E[\boldsymbol{p}^n]$ is finite. Then  $\E_a\left[{\bf Z}_{0<x}^n\right]$ is finite and satisfies the recursive equation
\begin{equation}\label{eq:EZxn}
\E_a\left[{\bf Z}_{0<x}^n\right]=\E_a\left[1_{\{\tau_0^- < \tau_x^+ \}} e^{\beta(m-1) \tau_0^-} \right] + \beta  \int_0^{+\infty} \E_{a}\left[1_{\{\tau_0^-\wedge \tau_x^+ > r \}} R_n(L_{r})  \right]  e^{\beta(m-1) r}dr
\end{equation}
where $R_1=0$ and for $n\geq2$ and $z\in(0,x)$, the remainder $R_n$ is defined by: 
\begin{equation}\label{eq:Rn}
R_n(z) := \sum_{k=1}^{+\infty}p_k \sum_{\substack{ i_1+\ldots+i_k=n\\ i_1\neq n,\ldots, i_k\neq n}} \binom{n}{i_1, \ldots, i_k}   \prod_{j=1}^k  \E_z\left[{\bf Z}_{0<x}^{i_j}\right].
\end{equation}
Furthermore $R_n$ is bounded by:
\begin{equation}\label{eq:boundRn}
R_n(z) \leq    \E\left[\boldsymbol{p}^n\right] \sum_{j=1}^{n-1}  \E_z\left[{\bf Z}_{0<x}^{n-j}\right] \E_z\left[{\bf Z}_{0<x}^{j}\right].
\end{equation}
\end{lemma}
\smallskip

\begin{proof}
We set for $\lambda>0$, $M_n^{(\lambda)}(a) := \E_a\left[{\bf Z}_{0<x}^ne^{-\lambda {\bf Z}_{0<x} }\right]$, and we denote by $\phi$ the initial particle.
Let us start by applying the Markov property at the first branching event:
\begin{align*}
M_n^{(\lambda)}(a) &= \E_a\left[{\bf Z}_{0<x}^ne^{-\lambda {\bf Z}_{0<x} } 1_{\{\phi \text{ has exited $(0,x)$ before splitting}\}}\right] \\
&\qquad \qquad \qquad + \E_a\left[{\bf Z}_{0<x}^ne^{-\lambda {\bf Z}_{0<x} }1_{\{\phi \text{ has stayed in $(0,x)$ before splitting} \}}\right]\\
&= e^{-\lambda}\Pb_a\left(\tau_0^- < \e \wedge \tau_x^+\right) +   \E_a\left[ 1_{\{\tau_0^-\wedge \tau_x^+ > {\bf e} \}}  \sum_{k=1}^{+\infty} p_k \E_{L_\e}\left[\left(\sum_{i=1}^k {\bf Z}^{(i)}_{0<x}\right)^n e^{-\lambda \sum_{i=1}^k {\bf Z}^{(i)}_{0<x}}   \right]\right]
\end{align*}
where $({\bf Z}^{(i)}_{0<x}, \, i\leq k)$ are  independent copies of ${\bf Z}_{0<x}$, and {\bf e} denotes an independent exponential random variable of parameter $\beta$. We now expand the power using the multinomial theorem:
$$
 \sum_{k=1}^{+\infty} p_k \E_{L_\e}\left[\left(\sum_{i=1}^k {\bf Z}^{(i)}_{0<x}\right)^n e^{-\lambda \sum_{i=1}^k {\bf Z}^{(i)}_{0<x}}   \right]= \sum_{k=1}^{+\infty} p_k  k M_n^{(\lambda)}(L_\e) \left(M_0^{(\lambda)}(L_\e)\right)^{k-1} +       R_n^{(\lambda)}(L_\e)
$$
where $R_{1}^{(\lambda)}=0$ and for $n\geq2$ and $z\in(0,x)$:
\begin{equation}\label{eq:Rnlambda}
R_n^{(\lambda)}(z) =   \sum_{k=1}^{+\infty} p_k \sum_{\substack{ i_1+\ldots+i_k=n\\ i_1\neq n,\ldots, i_k\neq n}} \binom{n}{i_1, \ldots, i_k}   \prod_{j=1}^k M_{i_j}^{(\lambda)}(z).
\end{equation}
Note that $R_n^{(0)}= R_n$ as defined in  (\ref{eq:Rn}). 
Bounding $M_0^{(\lambda)}$ by 1, this yields the formula
\begin{equation}\label{eq:Mnl}
M_n^{(\lambda)}(a) \leq e^{-\lambda}\Pb_a\left(\tau_0^- < \e \wedge \tau_x^+\right) + m  \E_a\left[ 1_{\{\tau_0^-\wedge \tau_x^+ > {\bf e} \}}  M_n^{(\lambda)}(L_\e) \right]  +   \E_a\left[ 1_{\{\tau_0^-\wedge \tau_x^+> {\bf e} \}}  
R_n^{(\lambda)}(L_\e)\right].
\end{equation}

Let $(\e_i)_{i\in \N}$ be a sequence of  independent exponential random variables with parameter $\beta$, independent of $L$. For  $k\in \N$, we set $\g_k := \sum_{i=1}^{k}\e_{i}$, with the convention that $\g_0:=0$. Using the Markov property, we deduce by an iterative procedure that for any $k\in \N$:
\begin{multline}\label{eq:boundrecMn}
M_n^{(\lambda)}(a)  \leq e^{-\lambda}\sum_{i=0}^k m^{i}\Pb_a(\tau_0^- < \tau_x^+, \g_i < \tau_0^- < \g_{i+1}) +  m^{k+1}  \E_a\left[ 1_{\{\tau_0^-\wedge \tau_x^+ > {\bf g}_{k+1} \}}  M_n^{(\lambda)}(L_{\g_{k+1}})\right]\\
+\sum_{i=0}^k m^{i}  \E_a\left[ 1_{\{\tau_0^-\wedge \tau_x^+>{\bf g}_{i+1} \}}  R_n^{(\lambda)}(L_{\g_{i+1}})\right]  .
\end{multline}
Recall now that the random variable $\g_{k+1}$ follows the Gamma distribution:
\begin{equation}\label{eq:gk+1}
\Pb\left(\g_{k+1}\in dr\right) = \frac{\beta^{k+1}}{k!} r^{k} e^{-\beta r}dr, \qquad r>0.
\end{equation}
Using the classical inequality $x^n e^{-\lambda x} \leq (n/\lambda)^n e^{-n}$ for $x,n,\lambda>0$, we have
\begin{equation}\label{eq:boundMnlambda}
M_n^{(\lambda)}(a) = \E_a\left[{\bf Z}_{0<x}^ne^{-\lambda {\bf Z}_{0<x} }\right] \leq   \left(\frac{n}{\lambda}\right)^ne^{-n}
\end{equation}
which is independent of $a$. This implies that the second term in (\ref{eq:boundrecMn}) is bounded by:
\begin{align}
\notag &m^{k+1}  \E_a\left[ 1_{\{\tau_0^-\wedge \tau_x^+ > {\bf g}_{k+1} \}}  M_n^{(\lambda)}(L_{\g_{k+1}})\right]\\
\notag  &\qquad\leq  m^{k+1}  \Pb_a\left(\tau_0^-\wedge \tau_x^+ > {\bf g}_{k+1}\right) \left(\frac{n}{\lambda}\right)^ne^{-n}\\
\notag &\qquad=\left(\frac{n}{\lambda}\right)^ne^{-n}    \int_0^{+\infty} \Pb_a\left(\tau_0^-\wedge \tau_x^+ > r\right) \frac{(m\beta)^{k+1}}{k!} r^{k} e^{-\beta r}dr  \\
\notag &\qquad=\left(\frac{n}{\lambda}\right)^ne^{-n}  \frac{(m\beta)^{k+1}}{k!}     \int_0^{1} \E_a\left[e^{\beta(m-1)(\tau_0^-\wedge \tau_x^+) s}  (\tau_0^-\wedge \tau_x^+)\times   (\tau_0^-\wedge \tau_x^+)^{k}e^{-m\beta(\tau_0^-\wedge \tau_x^+) s}\right] s^{k} ds  \\
\notag &\qquad\leq \left(\frac{n}{\lambda}\right)^ne^{-n}  \frac{(m\beta)^{k+1}}{k!}   \left(\frac{k}{m\beta}\right)^{k} e^{-k}    \int_0^{1}     \E_a\left[e^{\beta(m-1)(\tau_0^-\wedge \tau_x^+) s}  (\tau_0^-\wedge \tau_x^+)  \right] ds\\
\label{lim2term}&\qquad\leq  \left(\frac{n}{\lambda}\right)^ne^{-n}  \frac{k^k e^{-k}  }{k!} \frac{m}{m-1}  \E_a\left[e^{\beta(m-1)(\tau_0^-\wedge \tau_x^+) }  \right]\xrightarrow[k\rightarrow +\infty]{}0
\end{align}
from Stirling's asymptotics. Note that the exponential moments of $\tau_0^-\wedge \tau_x^+$ are finite thanks to (\ref{eq:asympT0}) since $\beta(m-1)\leq q_\ast$ and  $\tau_0^-\wedge \tau_x^+\leq \tau_0^-$ a.s.
 As a consequence, letting $k\rightarrow +\infty$ in (\ref{eq:boundrecMn}), we obtain the bound
$$ M_n^{(\lambda)}(a) \leq  e^{-\lambda }\sum_{i=0}^{+\infty} m^{i}\Pb_a(\tau_0^- < \tau_x^+, \g_i < \tau_0^- < \g_{i+1}) +
\sum_{i=0}^{+\infty} m^{i}  \E_a\left[ 1_{\{\tau_0^-\wedge \tau_x^+ > {\bf g}_{i+1} \}}  R_n^{(\lambda)}(L_{\g_{i+1}})\right]. $$
Recall next that conditionally to $\{\g_{i+1}=r\}$, the random variable $\g_i$ is distributed as the maximum of $i$ independent uniform random variables on $[0,r]$, i.e. 
$$\Pb(\g_i\in du|\g_{i+1}=r ) = \frac{i}{r^{i}} u^{i-1} du,\qquad \qquad u\in(0,r).$$
The sums may thus be computed and one obtains, using the Fubini-Tonelli theorem,  
\begin{align}
\notag&\sum_{i=0}^{+\infty} m^{i}\Pb_a(\tau_0^- < \tau_x^+, \g_i < \tau_0^- <\g_{i+1})\\
\notag&\qquad\qquad  = \int_0^{+\infty}\!\! \int_0^{r} \sum_{i=0}^{+\infty} m^{i}\Pb_a(\tau_0^- < \tau_x^+, u< \tau_0^- < r)   \frac{i}{r^{i}} u^{i-1} du  \frac{\beta^{i+1}}{i!} r^{i} e^{-\beta r}dr\\
\notag&\qquad\qquad  = \int_0^{+\infty}  \sum_{i=0}^{+\infty} m^{i} \E_a\left[1_{\{\tau_0^- < \tau_x^+, \tau_0^- < r\}} (\tau_0^-)^i \right]  \frac{\beta^{i+1}}{i!}  e^{-\beta r}dr\\
\notag&\qquad\qquad  =  \int_0^{+\infty}  \E_a\left[1_{\{\tau_0^- < \tau_x^+, \tau_0^-< r\}} e^{m \beta \tau_0^-} \right] \beta e^{-\beta r}dr\\
\label{lim1term}&\qquad\qquad  =\E_a\left[1_{\{\tau_0^- < \tau_x^+ \}} e^{\beta(m-1) \tau_0^-} \right], 
\end{align}
and similarly, using only (\ref{eq:gk+1}), 
\begin{equation}\label{limRnlambda}
\sum_{i=0}^{+\infty} m^{i}  \E_a\left[ 1_{\{\tau_0^-\wedge \tau_x^+ > {\bf g}_{i+1} \}}  R_n^{(\lambda)}(L_{\g_{i+1}})\right]   =\beta  \int_0^{+\infty} \E_{a}\left[1_{\{\tau_0^-\wedge \tau_x^+ > r \}} R_n^{(\lambda)}(L_{r})  \right]  e^{\beta(m-1) r}dr.
\end{equation}
As a consequence, we have thus proven that 
\begin{equation}\label{eq:Mnl<Rnl}
M_n^{(\lambda)}(a) \leq  e^{-\lambda }\E_a\left[1_{\{\tau_0^- < \tau_x^+ \}} e^{\beta(m-1) \tau_0^-} \right]  +\beta  \int_0^{+\infty} \E_{a}\left[1_{\{\tau_0^-\wedge \tau_x^+> r \}} R_n^{(\lambda)}(L_{r})  \right]  e^{\beta(m-1) r}dr.
\end{equation}
In particular, taking $n=1$ and recalling that $R_1^{(\lambda)}=0$, we deduce, by letting $\lambda \downarrow0$ and applying the monotone convergence theorem, that the expectation of ${\bf Z}_{0<x}$ is finite and uniformly bounded for $a\in(0,x)$:
$$ \E_a\left[{\bf Z}_{0<x}\right] =M_1^{(0)}(a)\leq \E_a\left[1_{\{\tau_0^- < \tau_x^+ \}} e^{\beta(m-1) \tau_0^-} \right]\leq   \E_x\left[e^{\beta(m-1) \tau_0^-} \right] <+\infty$$
where the finiteness of the last expectation comes from (\ref{eq:asympT0}) since $\beta(m-1)\leq q_\ast$.
We now take a closer look at $R_n^{(\lambda)}$ defined in (\ref{eq:Rnlambda}). Let us set $i_0(k) := \min\{i\in\{i_1, \ldots, i_k\}, i\geq1 \}$, i.e. $i_0(k)$ is the smallest non-null index in the sum.  Using  the FKG inequality
\begin{equation}\label{eq:ZZZ}
 \E_z\left[  \left({\bf Z}_{0<x}\right)^{i} \right]\E_z\left[ \left({\bf Z}_{0<x}\right)^{j} \right]\leq  \E_z\left[  \left({\bf Z}_{0<x}\right)^{i+j} \right], 
 \end{equation}
which is valid for any $i,j\geq0$ since ${\bf Z}_{0<x}$ is a non-negative random variable, we first have 
\begin{align*}
R_n^{(\lambda)}(z) \leq R_n(z)  \leq  \sum_{k=1}^{+\infty} p_k \sum_{\substack{ i_1+\ldots+i_k=n\\ i_1\neq n,\ldots, i_k\neq n}} \binom{n}{i_1, \ldots, i_k}   \E_z\left[{\bf Z}_{0<x}^{n-i_0(k)}\right] \E_z\left[{\bf Z}_{0<x}^{i_0(k)}\right]. 
\end{align*}
We next write
$$ \E_z\left[{\bf Z}_{0<x}^{n-i_0(k)}\right] \E_z\left[{\bf Z}_{0<x}^{i_0(k)}\right] \leq \sum_{j=1}^{n-1}  \E_z\left[{\bf Z}_{0<x}^{n-j}\right] \E_z\left[{\bf Z}_{0<x}^{j}\right]$$
which is now independent of the summation indices. As a consequence, applying the multinomial theorem, we obtain
$$R_n^{(\lambda)}(z) \leq R_n(z)  \leq  \sum_{k=1}^{+\infty} p_k k^n  \left( \sum_{j=1}^{n-1}  \E_z\left[{\bf Z}_{0<x}^{n-j}\right] \E_z\left[{\bf Z}_{0<x}^{j}\right]\right)$$
which proves the bound (\ref{eq:boundRn}) for $R_n$. Since ${\bf Z}_{0<x}$ is integer-valued, we may even write
\begin{equation}\label{eq:Rnfort}
 R_n(z)  \leq \E[\boldsymbol{p}^n] n  \left(\E_z\left[{\bf Z}_{0<x}^{n-1}\right]\right)^2.
 \end{equation}
Plugging (\ref{eq:Rnfort}) into (\ref{eq:Mnl<Rnl}) and letting $\lambda\downarrow0$ on the left-hand side, we now have
\begin{align}
\notag M_n^{(0)}(a)& \leq \E_a\left[e^{\beta(m-1) \tau_0^-} \right]  +\beta \E[\boldsymbol{p}^n] n \int_0^{+\infty} \E_{a}\left[1_{\{\tau_0^-\wedge \tau_x^+> r \}}  \left(\E_{L_r}\left[{\bf Z}_{0<x}^{n-1}\right]\right)^2   \right]  e^{\beta(m-1) r}dr\\
\label{eq:recur}&\leq  \E_x\left[e^{\beta(m-1) \tau_0^-} \right] +  \sup_{z\in(0,x)}  \left(\E_z\left[{\bf Z}_{0<x}^{n-1}\right]\right)^2\times  \frac{ \E[\boldsymbol{p}^n] n}{m-1}   \E_x\left[e^{\beta(m-1)\tau_0^- }  \right].
\end{align}
This implies by recursion  that $\displaystyle \sup_{a\in(0,x)}M_n^{(0)}(a) = \sup_{a\in(0,x)} \E_a[{\bf Z}_{0<x}^n]$ is finite for every $n\in\N$ such that $\E[\boldsymbol{p}^n]$ is finite.\\

It remains finally to obtain the recursive equation (\ref{eq:EZxn}). We start again from the beginning of the proof, this time taking $\lambda=0$:
$$
M_n^{(0)}(a) = \Pb_a\left(\tau_0^- < \e \wedge \tau_x^+\right) + m  \E_a\left[ 1_{\{\tau_0^-\wedge \tau_x^+ > {\bf e} \}}  M_n^{(0)}(L_\e)\right] + \E_a\left[ 1_{\{\tau_0^-\wedge \tau_x^+> {\bf e} \}} R_n(L_\e)\right].
$$
By iteration, we obtain:
\begin{multline}\label{eq:newMn}
M_n^{(0)}(a)  = \sum_{i=0}^k m^{i}\Pb_a(\tau_0^- < \tau_x^+, \g_i < \tau_0^- < \g_{i+1}) +  m^{k+1}  \E_a\left[ 1_{\{\tau_0^-\wedge \tau_x^+ > {\bf g}_{k+1} \}}  M_n^{(0)}(L_{\g_{k+1}})\right]\\+\sum_{i=0}^k m^{i}  \E_a\left[ 1_{\{\tau_0^-\wedge \tau_x^+>{\bf g}_{i+1} \}}  R_n(L_{\g_{i+1}})\right]  .
\end{multline}
From (\ref{lim1term}) and (\ref{limRnlambda}), the sums of the first and third terms in (\ref{eq:newMn}) converge as before as $k\rightarrow+\infty$ towards
$$\E_a\left[1_{\{\tau_0^- < \tau_x^+ \}} e^{\beta(m-1) \tau_0^-} \right] + \beta  \int_0^{+\infty} \E_{a}\left[1_{\{\tau_0^-\wedge \tau_x^+ > r \}} R_n(L_{r})  \right]  e^{\beta(m-1) r}dr$$
while the second term is shown to converge towards 0 in exactly the same way as in (\ref{lim2term}), replacing the bound (\ref{eq:boundMnlambda}) by
$$ 1_{\{\tau_0^-\wedge \tau_x^+ > {\bf g}_{k+1} \}}  M_n^{(0)}(L_{\g_{k+1}}) \leq   1_{\{\tau_0^-\wedge \tau_x^+ > {\bf g}_{k+1} \}}  \sup_{z\in(0,x)} M_n^{(0)}(z).$$

\end{proof}

\section{Asymptotics of the maximum {\bf M}} \label{sec:3}

This short section is devoted to the proof of Proposition \ref{prop:M}. We denote by ${\bf Z}_{x<0}$ the total number of individuals that enter $[x,+\infty)$ during the lifetime of the process, and whose ancestors have always stayed in $(0,x)$. The key observation is that
$$\Pb_a({\bf M}\geq x) = \Pb_a\left({\bf Z}_{x<0}\geq 1  \right)$$
so that we may write the bounds, using the Markov inequality and the Paley-Zygmund inequality, 
\begin{equation}\label{eq:boundsmax}
\Pb_a({\bf M}\geq x) \leq \E_a\left[{\bf Z}_{x<0}\right] \qquad \text{and}\qquad \Pb_a({\bf M}\geq x)  \geq \frac{ \left(\E_a\left[{\bf Z}_{x<0}\right]\right)^2 }{ \E_a\left[{\bf Z}_{x<0}^2\right] }.
\end{equation}
As a consequence, Proposition \ref{prop:M} will follow from the fact that the first and second moments of ${\bf Z}_{x<0}$ behave similarly as $x\rightarrow+\infty$.
\begin{lemma}\label{eq:asymmoment}
Assume that $a>0$ and  that $\E[\boldsymbol{p}^2]<+\infty$.
\begin{enumerate}[$(i)$]
\item The expectation of ${\bf Z}_{x<0}$ is asymptotically equivalent to 
$$\E_a\left[{\bf Z}_{x<0}\right]  \equi_{x\rightarrow+\infty}  
\begin{cases}
\displaystyle \Psi^\prime(\Phi_+(\beta(1-m))) W^{(\beta(1-m))}(a)  e^{-\Phi_+(\beta(1-m)) x} & \qquad \text{if } \beta(m-1)<q_\ast,\\
&\\
\displaystyle \frac{1}{2} \Psi^{\prime\prime}(\Phi(-q_\ast))  W^{(-q_\ast)}(a)  \frac{1}{x} e^{-\Phi(-q_\ast) x} & \qquad  \text{if }\beta(m-1)=q_\ast.
\end{cases}
$$
\item There exists a constant $C>0$ independent of $a>0$ such that as $x\rightarrow+\infty$
$$\E_a\left[{\bf Z}_{x<0}^2\right] \leq \begin{cases}
\displaystyle C\,  W^{(\beta(1-m))}(a) e^{-\Phi_+(\beta(1-m)) x} & \qquad \text{if } \beta(m-1)<q_\ast,\\
&\\
\displaystyle C  \, W^{(-q_\ast)}(a)  \frac{1}{x} e^{-\Phi(-q_\ast) x} & \qquad  \text{if }\beta(m-1)=q_\ast.
\end{cases}
$$
\end{enumerate}
\end{lemma}
\begin{proof}
To study the moments of ${\bf Z}_{x<0}$, we may simply exchange the roles of $0$ and $x$ in Lemma \ref{lem:Mn}. This yields:
\begin{equation}\label{eqZx0}
\E_a\left[{\bf Z}_{x<0}\right] = \E_a\left[1_{\{\tau_x^+ < \tau_0^- \}} e^{\beta(m-1) \tau_x^+} \right] 
\end{equation}
and
\begin{equation}\label{eqZx02}
\E_a\left[{\bf Z}_{x<0}^2\right] = \E_a\left[1_{\{\tau_x^+ < \tau_0^- \}} e^{\beta(m-1) \tau_x^+} \right] + \beta  \int_0^{+\infty} \E_{a}\left[1_{\{\tau_0^-\wedge \tau_x^+ > r \}} \widehat{R}_2(L_{r})  \right]  e^{\beta(m-1) r}dr
\end{equation}
where the remainder $\widehat{R}_2$ is bounded by:
\begin{equation}\label{eq:R2}
\widehat{R}_2(z) =  2 \sum_{k=1}^{+\infty} p_k \sum_{\substack{ i_1+\ldots+i_k=2\\ i_1\neq 2,\ldots, i_k\neq 2}}  \prod_{j=1}^k \E_z\left[{\bf Z}_{x<0}^{i_j}\right] \leq 2 \E[\boldsymbol{p}^2] \E_z\left[{\bf Z}_{x<0}\right]^2 .
\end{equation}
We now study the limit of both terms as $x\rightarrow +\infty$. First, from Kyprianou \cite[Theorem 8.1]{Kyp}, analytic continuation and (\ref{eq:W+=W-}), we have
\begin{equation}\label{eq:etaux}
\E_a\left[{\bf Z}_{x<0}\right] =\E_a\left[1_{\{\tau_x^+ < \tau_0^- \}} e^{\beta(m-1) \tau_x^+} \right] = \frac{W^{(\beta(1-m))}(a)}{W^{(\beta(1-m))}(x)} = e^{-\Phi_+(\beta(1-m))x} \frac{W^{(\beta(1-m))}(a)}{W_{\Phi_+(\beta(1-m))}(x)}.
\end{equation}
Point $(i)$ of Lemma \ref{eq:asymmoment} when  $\beta(m-1)<q_\ast$ then follows from the limit (\ref{eq:limWbeta}). On the other hand,  when $\beta(m-1)=q_\ast$, then $\Psi^\prime(\Phi(-q_\ast))=0$, and Point $(i)$ follows by replacing (\ref{eq:limWbeta}) by  (\ref{eq:limWq}). \\

We now examine the second moment of ${\bf Z}_{x<0}$. From (\ref{eqZx02}), (\ref{eq:R2}),  (\ref{eq:etaux}) and  (\ref{eq:W+=W-}), we have
\begin{multline*}
\E_a\left[{\bf Z}_{x<0}^2\right] \leq \E_a\left[1_{\{\tau_x^+ < \tau_0^- \}} e^{\beta(m-1) \tau_x^+} \right] + 2\E[\boldsymbol{p}^2]  \beta  \frac{e^{-2\Phi_-(\beta(1-m))x}}{W^2_{\Phi_-(\beta(1-m))}(x)}
\\\times   \int_0^{+\infty}  \E_{a}\left[1_{\{\tau_0^-\wedge \tau_x^+ > r \}}   e^{2\Phi_-(\beta(1-m))L_r } \left(W_{\Phi_-(\beta(1-m))}(L_r)\right)^2  \right]  e^{\beta(m-1) r}dr.
\end{multline*}
The first term on the right-hand side, which equals $\E_a\left[{\bf Z}_{x<0}\right]$,  has the right asymptotics, so it remains  to study the limit as $x\rightarrow+\infty$ of the integral, say $I(x)$. Using the Esscher transform, $I(x)$ reads
$$I(x)= e^{ \Phi_-(\beta(1-m))a}\int_0^{+\infty}  \E_{a}^{\Phi_-(\beta(1-m))}\left[1_{\{\tau_0^-  \wedge \tau_x^+  > r \}}  e^{ \Phi_-(\beta(1-m)) L_r} \left(W_{\Phi_-(\beta(1-m))}(L_r)\right)^2  \right]  dr.$$
Recall next, since the distribution of $L_t$ is assumed to be absolutely continuous for $t>0$,  that from Bertoin \cite[Theorem 1]{BerScale} the resolvent of the killed process on $(0, x)$ is defined by 
\begin{equation}\label{eq:resolvent}
\int_0^{+\infty}\Pb_a(L_t \in dy, \tau_0^- \wedge \tau_{x}^+ > t ) dt = \left( \frac{W(a)W(x-y)}{W(x)} - 1_{\{a\geq y\}} W(a-y)\right) dy.
\end{equation}
As a consequence, since $W_{\Phi_-(\beta(1-m))}$ is non-negative,  we obtain the bound
\begin{align*}
I(x) &\leq  e^{ \Phi_-(\beta(1-m))a}\int_0^{x}   e^{ \Phi_-(\beta(1-m)) y} \left(W_{\Phi_-(\beta(1-m))}(y)\right)^2  \frac{W_{\Phi_-(\beta(1-m))}(a)W_{\Phi_-(\beta(1-m))}(x-y)}{W_{\Phi_-(\beta(1-m))}(x)}  dy.
\end{align*}
Using the change of variables $z=x-y$, we further obtain, thanks to (\ref{eq:W+=W-}) and the fact that $W_{\Phi_+(\beta(1-m))}$ is increasing,
\begin{align*} 
\frac{e^{-\Phi_-(\beta(1-m))x}}{W_{\Phi_-(\beta(1-m))}(x)} I(x) &\leq W^{(\beta(1-m))}(a) \int_0^{x}    e^{ -\Phi_-(\beta(1-m)) z}  \left(\frac{W_{\Phi_-(\beta(1-m))}(x-z)}{W_{\Phi_-(\beta(1-m))}(x)}  \right)^2  W_{\Phi_-(\beta(1-m))}(z)dz\\
&\leq W^{(\beta(1-m))}(a) \int_0^{x}    e^{ -\Phi_+(\beta(1-m)) z}   \left(\frac{W_{\Phi_+(\beta(1-m))}(x-z)}{W_{\Phi_+(\beta(1-m))}(x)}  \right)^2    W_{\Phi_+(\beta(1-m))}(z)dz\\
&\leq W^{(\beta(1-m))}(a) \int_0^{+\infty}    e^{ -\Phi_+(\beta(1-m)) z}      W_{\Phi_+(\beta(1-m))}(z)dz.
\end{align*}
As a consequence, we deduce from (\ref{eq:etaux}) that for any $a>0$,
\begin{equation}\label{eq:boundZx2}
 \E_a\left[{\bf Z}_{x<0}^2\right] \leq W^{(\beta(1-m))}(a)  \frac{ e^{-\Phi_+(\beta(1-m))x}}{W_{\Phi_+(\beta(1-m))}(x)} \left( 1 +   2\E[\boldsymbol{p}^2]  \beta  \int_0^{+\infty}    e^{ -\Phi_+(\beta(1-m)) z}      W_{\Phi_+(\beta(1-m))}(z)dz\right)
 \end{equation}
which proves Point $(ii)$.
\end{proof}
\noindent
Finally, the proof of Proposition \ref{prop:M} follows by plugging the asymptotics given in Lemma \ref{eq:asymmoment} into (\ref{eq:boundsmax}).


\section{The subcritical case}\label{sec:4}

We start by studying the subcritical case $\beta(m-1)<q_\ast$. Note that from Lemma \ref{lem:Mn}, we have
\begin{equation}\label{eq:EZx}
\E_a\left[{\bf Z}_{0<x}\right]=\E_a\left[1_{\{\tau_0^- < \tau_x^+ \}} e^{\beta(m-1) \tau_0^-} \right].
\end{equation}
Taking the limit as $x\rightarrow +\infty$, we deduce that
$$\E_a\left[{\bf Z}_0\right]=\E_a\left[ e^{\beta(m-1) \tau_0^-} \right]$$
which is finite thanks to (\ref{eq:asympT0}). Applying the Esscher transform (\ref{eq:Esscher}) with $c=\Phi_-(\beta(1-m))$ together with the optional sampling theorem, we first obtain the identity  
\begin{align*}
 \E_a\left[e^{\beta(m-1) (t\wedge \tau_0^{-})} \right]&= e^{\Phi_-(\beta(1-m))a}\E_{a}^{\Phi_-(\beta(1-m))}\left[e^{-\Phi_-(\beta(1-m))L_{t\wedge \tau_0^-}}  \right]\\
 &= e^{\Phi_-(\beta(1-m))a}\E_{a}^{\Phi_-(\beta(1-m))}\left[e^{-\Phi_-(\beta(1-m))L_{ \tau_0^-}}1_{\{\tau_0^- \leq t\}}  \right]\\
 &\qquad\qquad\qquad\qquad + e^{\Phi_-(\beta(1-m))a}\E_{a}^{\Phi_-(\beta(1-m))}\left[e^{-\Phi_-(\beta(1-m))L_{t}}1_{\{\tau_0^- >t\}}  \right].
 \end{align*}
Note that $\tau_0^-$ is a.s. finite under $\Pb_a^{\Phi_-(\beta(1-m))}$ since $\Psi_{\Phi_-(\beta(1-m))}^\prime(0^+) = \Psi^\prime(\Phi_-(\beta(1-m)))<0$. Also, since by definition $L_t\geq0$ for $t< \tau_0^-$, we have
$$\E_{a}^{\Phi_-(\beta(1-m))}\left[e^{-\Phi_-(\beta(1-m))L_{t}}1_{\{\tau_0^- >t\}}  \right]\leq \Pb_a^{\Phi_-(\beta(1-m))}(\tau_0^- >t)\xrightarrow[t\rightarrow+\infty]{}0.$$
As a consequence, applying the monotone convergence theorem, we obtain the identity 
\begin{equation}\label{eq:Etau0}
  \E_a\left[e^{\beta(m-1)\tau_0^{-}} \right] =  e^{\Phi_-(\beta(1-m))a}\E_{a}^{\Phi_-(\beta(1-m))}\left[e^{-\Phi_-(\beta(1-m))L_{\tau_0^-}}  \right].
  \end{equation}
 We thus deduce that, although $L_{\tau_0^-}\leq 0$ a.s.,  the expectation of $\exp\left(-\Phi_-(\beta(1-m))L_{\tau_0^-}\right) $ under $\Pb_a^{\Phi_-(\beta(1-m))}$ is finite. 
 In fact, the following uniform bound is proven in Lemma \ref{lem:C1} in the Appendix:
 \begin{equation}\label{eq:SupL0Text}
 \sup_{a\geq0} \E_{a}^{\Phi_-(\beta(1-m))}\left[e^{-\Phi_-(\beta(1-m))L_{\tau_0^-}}  \right] <+\infty.
 \end{equation}
The following lemma gives a control on all the moments of the random variable ${\bf Z}_{0<x}$.

\begin{lemma}\label{lem:Z0b}
Set $\gamma = \left\lfloor \frac{\Phi_+(\beta(1-m))}{\Phi_-(\beta(1-m))}\right\rfloor\geq 1$ and recall that $\E[\boldsymbol{p}^{\gamma+1}]<+\infty$.
We assume that $0<a<x$.
\begin{enumerate}[(i)]
\item For any $k\in \N$ such that $k< \gamma$, there exists a constant $C_k$, independent of $x$ and $a$, such that 
$$\E_a\left[ {\bf Z}_{0<x}^{k}\right] \leq C_k e^{k \Phi_-(\beta(1-m))a}.$$
\item For $k=\gamma$, there exists a constant $C_\gamma$, independent of $x$ and $a$, such that 
$$\E_a\left[ {\bf Z}_{0<x}^{\gamma}\right] \leq C_\gamma (1+ x-a) e^{\gamma \Phi_-(\beta(1-m))a}.$$
\item For $k=\gamma+1$, there exists a constant $C_{\gamma +1}$, independent of $x$ and $a$, such that 
$$ \E_a\left[ {\bf Z}_{0<x}^{\gamma+1}\right] \leq  C_1 e^{\Phi_-(\beta(1-m)) a} +  C_{\gamma+1} \frac{W^{(\beta(1-m))}(a)}{W^{(\beta(1-m))}(x)} e^{(\gamma+1) \Phi_-(\beta(1-m)) x}.$$ 
\end{enumerate}
\end{lemma}

\begin{proof}
We prove Point $(i)$ by iteration. For $k=1$, this is a direct consequence of (\ref{eq:EZx}),  (\ref{eq:Etau0}) and (\ref{eq:SupL0Text}) since $ \E_a\left[{\bf Z}_{0<x}\right] \leq  \E_a\left[ {\bf Z}_{0}\right]$. 
Let us take $2\leq k\leq  \gamma$, and assume that the result holds true for all integers up to $k-1$.
From the bound (\ref{eq:boundRn}) of $R_k$ which only depends on the first  $(k-1)$-moments of ${\bf Z}_{0<x}$, we have 
$$R_{k}(z) \leq  A_k   e^{k\Phi_-(\beta(1-m))z} $$
where $A_k =  \E[\boldsymbol{p}^k] \sum_{j=1}^{k-1}C_j C_{k-j} $.
Then, using Lemma \ref{lem:Mn} and applying the Esscher transform with $c=k\Phi_-(\beta(1-m))$, we obtain
\begin{align}
\notag\E_a\left[ {\bf Z}_{0<x}^{k}\right] &\leq  \E_a\left[ {\bf Z}_{0<x}\right]  + A_k  \beta\int_0^{+\infty} \E_{a}\left[1_{\{\tau_0^-\wedge \tau_x^+ > r \}}  e^{k\Phi_-(\beta(1-m)) L_r }  \right]  e^{\beta(m-1) r}dr\\
\notag & \leq C_1 e^{\Phi_-(\beta(1-m)) a} \\
\label{eq:Ak1}&\qquad+A_k e^{k \Phi_-(\beta(1-m)) a  } \beta \int_0^{+\infty}    \E_{a}^{k\Phi_-(\beta(1-m))}\left[1_{\{\tau_0^-\wedge \tau_x^+ > r \}} \right] e^{(\beta(m-1) + \Psi(k\Phi_-(\beta(1-m)))) r} dr. 
\end{align}
Now, if $k< \gamma$, then $\Psi(k\Phi_-(\beta(1-m)))<\beta(1-m)$ and bounding the indicator function by one, the remaining integral is finite which yields Point $(i)$. To prove Point $(ii)$, observe first that when $k=\gamma$, the previous argument still applies if $ \frac{\Phi_+(\beta(1-m))}{\Phi_-(\beta(1-m))}$ is not an integer, i.e. if $ \frac{\Phi_+(\beta(1-m))}{\Phi_-(\beta(1-m))}\neq \gamma$. Otherwise, we have $\Psi(\gamma\Phi_-(\beta(1-m))) = \Psi(\Phi_+(\beta(1-m))) = \beta(1-m)$, i.e. the exponential cancels, and the integral in (\ref{eq:Ak1}) becomes 
$$\int_0^{+\infty}    \E_{a}^{\Phi_+(\beta(1-m))}\left[1_{\{\tau_0^-\wedge \tau_x^+ >r \}} \right]  dr =  \E_a^{\Phi_+(\beta(1-m))}\left[\tau_0^-\wedge \tau_x^+\right] \leq \E_a^{\Phi_+(\beta(1-m))}\left[ \tau_x^+\right] . $$
 Note that under $\Pb_a^{\Phi_+(\beta(1-m))}$ the process $(\tau_z^+, z\geq0)$ is a subordinator with Laplace exponent $\Phi_+(\beta(1-m)+\lambda)- \Phi_+(\beta(1-m))$.
  Using \cite[Chapter VII, Theorem 1]{Ber} and the invariance by translation, its expectation equals
  \begin{equation}\label{eq:tauz}
  \E_{a}^{\Phi_+(\beta(1-m))}\left[\tau_{x}^+\right] =  \frac{x-a}{\Psi^\prime(\Phi_+(\beta(1-m)))} 
  \end{equation}
which proves Point $(ii)$ in this case. Finally, to prove Point $(iii)$, we start again from Lemma \ref{lem:Mn}. Proceeding as in the proof of Point $(ii)$ and using the bounds from Points $(i)-(ii)$, we have
$$R_{\gamma+1}(z) \leq A_{\gamma+1} (1+x-z) e^{(\gamma+1) \Phi_-(\beta(1-m)) z}$$
for some constant $A_{\gamma+1}>0$, and we obtain, using the Esscher transform  with $c=\Phi_+(\beta(1-m))$: 
\begin{align}
\notag \E_a\left[ {\bf Z}_{0<x}^{\gamma+1}\right] &\leq   \E_a\left[ {\bf Z}_{0<x}\right]  +A_{\gamma+1}  \beta  \int_0^{+\infty} \E_{a}\left[1_{\{\tau_0^-\wedge \tau_x^+ > r \}}  \left(1+  x-L_r\right)e^{(\gamma+1)\Phi_-(\beta(1-m)) L_r }  \right]  e^{\beta(m-1) r}dr\\
\notag &\leq C_1 e^{\Phi_-(\beta(1-m)) a} +A_{\gamma+1}  \beta  e^{\Phi_+(\beta(1-m)) a}   \times \\
\notag&\qquad  \int_0^{+\infty} \E_{a}^{\Phi_+(\beta(1-m))}\left[1_{\{\tau_0^-\wedge \tau_x^+ > r \}}  \left(1+ x-L_r\right)e^{((\gamma+1)\Phi_-(\beta(1-m)) -\Phi_+(\beta(1-m))) L_r }  \right] dr.
\end{align}
Using the expression for the resolvent of the killed process (\ref{eq:resolvent}), a change of variables and (\ref{eq:esscherparam}), the remaining integral is bounded by 
\begin{multline*}
\int_0^{x}   \left(1+ x-y\right)e^{((\gamma+1)\Phi_-(\beta(1-m)) -\Phi_+(\beta(1-m)) y}   \frac{W_{\Phi_+(\beta(1-m))}(a)W_{\Phi_+(\beta(1-m))}(x-y)}{W_{\Phi_+(\beta(1-m))}(x)} dy\\
= e^{(\gamma+1)\Phi_-(\beta(1-m))x}   \frac{W_{\Phi_+(\beta(1-m))}(a)}{W^{(\beta(1-m))}(x)}  \int_0^{x}   \left(1+ z \right)e^{  -((\gamma+1)\Phi_-(\beta(1-m)) -\Phi_+(\beta(1-m)) z}  W_{\Phi_+(\beta(1-m))}(z)dz.
\end{multline*}
The claim finally follows from the fact that the limit as $x\rightarrow +\infty$ of the last integral is finite since $(\gamma+1)\Phi_-(\beta(1-m)) >\Phi_+(\beta(1-m))$ and $W_{\Phi_+(\beta(1-m))}$ is bounded from (\ref{eq:limWbeta}).
\end{proof}

\subsection{Upper bound}

Let us recall that we denote by ${\bf M}$ the overall maximum of the branching process $X$, and that its asymptotics are given in Proposition \ref{prop:M}.
We decompose the tail of ${\bf Z}_0$ according to whether ${\bf M}$ is large or not :
\begin{align}
\notag \Pb_a({\bf Z}_0 \geq n) &\leq \Pb_a({\bf Z}_0 \geq n, {\bf M}< x) +  \Pb_a( {\bf M}\geq x) \\
\notag & \leq \Pb_a({\bf Z}_{0<x}\geq n) +  \Pb_a( {\bf M}\geq x)  \\
\label{eq:UpB}&\leq n^{-\gamma-1}  \E_a\left[{\bf Z}_{0<x}^{\gamma+1}\right]+  \Pb_a( {\bf M}\geq x) .
\end{align}
By assumption, $\E[\boldsymbol{p}^{\gamma+1}]<+\infty$, hence from Lemma \ref{lem:Mn}, the expectation of ${\bf Z}_{0<x}^{\gamma+1}$ is finite.
Consider the sequence $(x_n)_{n\geq 1}$ defined by 
$$x_n = \frac{1}{\Phi_-(\beta(1-m))}\ln(n).$$
From (\ref{eq:boundsmax}) and (\ref{eq:etaux}), we have
$$\Pb_a( {\bf M}\geq x_n) \leq  \frac{W^{(\beta(1-m))}(a)}{W_{\Phi_+(\beta(1-m))}(x_n)}e^{-\Phi_+(\beta(1-m)) x_n}  = \frac{ W^{(\beta(1-m))}(a)}{W_{\Phi_+(\beta(1-m))}(x_n)}  n^{- \frac{\Phi_+(\beta(1-m))}{\Phi_-(\beta(1-m))}}$$
while from Lemma \ref{lem:Z0b},
$$
 n^{-\gamma-1}  \E_a\left[{\bf Z}_{0<x_n}^{\gamma+1}\right] \leq  n^{-\gamma-1} C_1 e^{\Phi_-(\beta(1-m))a} + C_{\gamma+1} \frac{ W^{(\beta(1-m))}(a)}{W_{\Phi_+(\beta(1-m))}(x_n)}  
n^{- \frac{\Phi_+(\beta(1-m))}{\Phi_-(\beta(1-m))}}.$$
Recall next that $a\geq 1$. Using (\ref{eq:esscherparam}) and the fact that $W_{\Phi_-(\beta(1-m))}$ is increasing, we have 
\begin{align*}
n^{ \frac{\Phi_+(\beta(1-m))}{\Phi_-(\beta(1-m))}}\Pb_a({\bf Z}_0 \geq n) &\leq W^{(\beta(1-m))}(a) \left( \frac{C_1}{W_{\Phi_-(\beta(1-m))}(1)}  n^{\frac{\Phi_+(\beta(1-m))}{\Phi_-(\beta(1-m))}-1-\gamma} +\frac{1+ C_{\gamma+1}}{W_{\Phi_+(\beta(1-m))}(x_n)}  \right)\\
& \leq \kappa_2  W^{(\beta(1-m))}(a)
\end{align*}
for some constant $\kappa_2>0$.

\subsection{Lower bound}
We now proceed to the lower bound. Applying the strong Markov property 
\begin{align}
\notag \Pb_a({\bf Z}_0 \geq n) &\geq \Pb_a\left({\bf Z}_0 \geq n| {\bf M}\geq x \right) \Pb_a({\bf M}\geq x )\\
\notag & \geq \Pb_x\left({\bf Z}_0 \geq n \right) \Pb_a({\bf M}\geq x )\\
\label{eq:LowB}&\geq \Pb_x\left({\bf Z}_{0<x+k} \geq n \right) \Pb_a({\bf M}\geq x )
\end{align}
for some $k>1$ large enough to be chosen later.
Furthermore, from (\ref{eq:EZx}), the absence of positive jumps together with the Markov property allow to write down the decomposition 
\begin{align*}
\E_x\left[ {\bf Z}_{0<x+k}\right] &=   \E_x\left[ e^{\beta(m-1) \tau_0^-} \right] - \E_x\left[1_{\{\tau_{x+k}^+ < \tau_0^-\}} e^{\beta(m-1) \tau_{x+k}^+}\right] \E_{x+k}\left[e^{\beta (m-1)\tau_0^-}\right]\\
&= e^{\Phi_-(\beta(1-m)) x} \E_x^{\Phi_-(\beta(1-m))}\left[ e^{-\Phi_-(\beta(1-m))L_{\tau_0^-}} \right] \\
&\qquad \quad -\E_x\left[1_{\{\tau_{x+k}^+ < \tau_0^-\}} e^{\beta(m-1) \tau_{x+k}^+}\right] e^{\Phi_-(\beta(1-m)) (x+k)} \E_{x+k}^{\Phi_-(\beta(1-m))}\left[ e^{-\Phi_-(\beta(1-m))L_{\tau_0^-}} \right] .
\end{align*}
where we used (\ref{eq:Etau0}) in the second equality. From (\ref{eq:etaux}), (\ref{eq:W+=W-}) and the fact that $W_{\Phi_+(\beta(1-m))}$ is non-negative and increasing, we deduce that 
\begin{equation}\label{eq:W/W}
\E_x\left[1_{\{\tau_{x+k}^+ < \tau_0^-\}} e^{\beta(m-1) \tau_{x+k}^+}\right] = e^{-k\Phi_+(\beta(1-m))} \frac{W_{\Phi_+(\beta(1-m))}(x)}{W_{\Phi_+(\beta(1-m))}(x+k)}\leq  e^{-k\Phi_+(\beta(1-m))}   
\end{equation}
As a consequence, since $\beta(m-1)<q_\ast$, we obtain from Lemma \ref{lem:C1} in the Appendix the bounds
\begin{equation}\label{eq:ExZx}
 e^{\Phi_-(\beta(1-m))x}\left(1 -  C_1 e^{-k(\Phi_+(\beta(1-m)) - \Phi_-(\beta(1-m))) } \right) \leq   \E_x\left[ {\bf Z}_{0<x+k}\right] \leq e^{\Phi_-(\beta(1-m))x} C_1
\end{equation}
for some constant $C_1>0$ independent of $k$, and we choose $k$ large enough for the left-hand side to be positive.
Let us take a sequence $(x_n)_{n\geq 1}$ such that $\E_{x_n}\left[ {\bf Z}_{0<x_n+k}\right]=2n$. From  (\ref{eq:ExZx}), we thus have
\begin{equation}\label{eq:xnbeta}
 \frac{1}{\Phi_-(\beta(1-m))}\ln(n) - c_1 \leq x_n\leq   \frac{1}{\Phi_-(\beta(1-m))}\ln(n) +c_2
 \end{equation}
for some positive constants $c_1, c_2$.
Applying Paley-Zygmund's inequality, we obtain
$$\Pb_{x_n}\left({\bf Z}_{0<x_n+k} \geq n\right) = \Pb_{x_n}\left({\bf Z}_{0<x_n+k} \geq \frac{1}{2} \E_{x_n}[{\bf Z}_{0<x_n+k}]\right) 
\geq \frac{n^2}{\E_{x_n}[{\bf Z}_{0<x_n+k}^2]}$$
and, going back to (\ref{eq:LowB}), we deduce from the asymptotics (\ref{eq:Mbeta}) of ${\bf M}$ that there exists a constant $\kappa_1>0$ independent of $a>0$ such that, as $n\rightarrow+\infty$, 
\begin{equation}\label{eq:lowB2}
 \Pb_a({\bf Z}_0 \geq n) \geq \kappa_1 \frac{n^2}{\E_{x_n}[{\bf Z}_{0<{x_n}+k}^2]}   W^{(\beta(1-m))}(a) n^{-\frac{\Phi_+(\beta(1-m))}{\Phi_-(\beta(1-m))}}.
 \end{equation}
It remains thus to study the asymptotics of the second moment of ${\bf Z}_{0<x_n+k}$ along the sequence $(x_n)_{n\geq 1}$. We rely on the bounds given in (\ref{eq:xnbeta}). From Lemma \ref{lem:Z0b}, three situations may occur:
\begin{enumerate}
\item If $\gamma=1$, then from Point $(iii)$ and (\ref{eq:W+=W-}), for $n$ large enough:
\begin{align*}
\E_{x_n}[{\bf Z}_{0<x_n+k}^2]  &\leq C_1 e^{\Phi_-(\beta(1-m)) x_n} +   C_2 \frac{W^{(\beta(1-m))}(x_n)}{W^{(\beta(1-m))}(x_n+k)} e^{2 \Phi_-(\beta(1-m))(x_n+k)}\\
&\leq C_1 e^{c_2\Phi_-(\beta(1-m)) } n +  C_2 e^{-\Phi_+(\beta(1-m)) k }  e^{2\Phi_-(\beta(1-m)) k }  e^{ 2c_2\Phi_-(\beta(1-m)) } n^2 \leq C_0 n^2.
\end{align*}
\item If $\gamma=2$, then from Point $(ii)$:
$$
\E_{x_n}[{\bf Z}_{0<x_n+k}^2]   \leq  C_2\left(1+k\right) e^{2 \left(\ln(n) + c_2\Phi_-(\beta(1-m)) \right)  }   \leq C_0 n^2.
$$
\item If $\gamma\geq 3$, then from Point $(i)$:
$$\E_{x_n}[{\bf Z}_{0<x_n+k}^2]  \leq  C_{2} e^{2 \left(\ln(n) + c_2\Phi_-(\beta(1-m)) \right)  }   \leq C_0 n^2.
 $$
\end{enumerate}
Finally, in all three cases, we have obtained that there exists a constant $C_0>0$ such that, as $n\rightarrow +\infty$,
$$\frac{n^2}{\E_{x_n}[{\bf Z}_{0<x_n+k}^2]} \geq \frac{1}{C_0}>0.$$
Plugging this last relation in (\ref{eq:lowB2}) gives the announced lower bound.
\qed

\section{The critical case}\label{sec:5}

We now turn our attention to the critical case. As for (\ref{eq:Etau0}) in the subcritical case, Lemma \ref{lem:Mn} and the Esscher transform (\ref{eq:Esscher}) imply
\begin{equation}\label{eq:EZq}
\E_{a}[{\bf Z}_0] = \E_{a}\left[e^{q_{\ast}\tau_0^-}\right] = e^{\Phi(-q_\ast) a} \E_a^{\Phi(-q_\ast) }\left[e^{- \Phi(-q_\ast) L_{\tau_0^-} }\right] <\infty
\end{equation}
which is finite thanks to (\ref{eq:asympT0}). We begin by collecting several properties of the function $W_{\Phi(-q_\ast)}$ that we shall need in the sequel.

\subsection{The function $W_{\Phi(-q_\ast)}$}
As pointed out in Remark \ref{rem:concave},  the concavity assumption of Theorem \ref{theo:1} is equivalent to the assumption that the function $W_{\Phi(-q_\ast)}$ is concave for $z\geq0$. Recall also that for $z<0$, we have by definition $W_{\Phi(-q_\ast)}(z) =0$. 
\begin{lemma}\label{lem:Wq}
Let $y>0$. 
\begin{enumerate}[(i)]
\item The function $x\longmapsto W_{\Phi(-q_\ast)}(x)-W_{\Phi(-q_\ast)}(x-y)$ is non-increasing for $x\geq y$ and it holds 
$$ \lim_{x\rightarrow +\infty}W_{\Phi(-q_\ast)}(x)-W_{\Phi(-q_\ast)}(x-y)=\frac{2y}{\Psi^{\prime\prime}(\Phi(-q_\ast))}.$$
\item  There exist two constants $c_1, c_2>0$ such that for $x\geq y$:
$$0\leq W_{\Phi(-q_\ast)}(x)-W_{\Phi(-q_\ast)}(x-y) - \frac{2y}{\Psi^{\prime\prime}(\Phi(-q_\ast))} \leq \frac{y}{x}(c_1+c_2y).$$
\end{enumerate}
\end{lemma}

\begin{proof}
The fact that the function $x\longmapsto W_{\Phi(-q_\ast)}(x)-W_{\Phi(-q_\ast)}(x-y)$ is non-increasing for $x\geq y$ is a consequence of the concavity assumption of $W_{\Phi(-q_\ast)}$. Then, starting from (\ref{eq:W}), we have 
\begin{equation}\label{eq:LapWW}
\int_0^{+\infty} e^{-\lambda x} \left(W_{\Phi(-q_\ast)}(x)-W_{\Phi(-q_\ast)}(x-y) \right) dx = \frac{1-e^{-\lambda y}}{\Psi(\lambda+\Phi(-q_\ast))+q_\ast} \equi_{\lambda\rightarrow 0}  \frac{2y}{\Psi^{\prime\prime}(\Phi(-q_\ast))\lambda}
\end{equation}
and the limit in  Point $(i)$ follows from Karamata's Tauberian theorem. Point $(i)$ further implies that for $a\geq y$, the function $a\rightarrow W_{\Phi(-q_\ast)}(a)-W_{\Phi(-q_\ast)}(a-y)-\frac{2y}{\Psi^{\prime\prime}(\Phi(-q_\ast))}$ is non-increasing and non-negative. An application of the monotone convergence theorem then shows that its integral over $(y,+\infty)$ is finite. Indeed, we have:
\begin{align*}
&Q(y):= \int_y^{+\infty}\left(W_{\Phi(-q_\ast)}(a)-W_{\Phi(-q_\ast)}(a-y) - \frac{2y}{\Psi^{\prime\prime}(\Phi(-q_\ast))}\right) da\\
&\qquad = \lim_{\lambda\downarrow 0} \int_y^{+\infty} e^{-\lambda a} \left(W_{\Phi(-q_\ast)}(a)-W_{\Phi(-q_\ast)}(a-y) - \frac{2y}{\Psi^{\prime\prime}(\Phi(-q_\ast))}\right) da\\
&\qquad =  \lim_{\lambda\downarrow 0}  \int_0^{+\infty} e^{-\lambda a} \left(W_{\Phi(-q_\ast)}(a)-W_{\Phi(-q_\ast)}(a-y) - \frac{2y}{\Psi^{\prime\prime}(\Phi(-q_\ast))}\right) da \\
&\qquad \qquad\qquad  -  \int_0^{y} e^{-\lambda a} W_{\Phi(-q_\ast)}(a) da  +  \frac{2y}{ \Psi^{\prime\prime}(\Phi(-q_\ast))}  \frac{ 1-e^{-\lambda y}}{\lambda}. 
\end{align*}
From (\ref{eq:LapWW}),  the limit of the integral on $(0,+\infty)$ is given by:
\begin{align*}
& \lim_{\lambda\downarrow 0}  \frac{\lambda y - \frac{\lambda^2}{2}y^2 + o(\lambda^2)}{
\frac{\lambda^2}{2} \Psi^{\prime\prime}(\Phi(-q_\ast)) + \frac{\lambda^3}{6} \Psi^{\prime\prime\prime}(\Phi(-q_\ast) )+ o(\lambda^3)} -  \frac{2y}{\lambda \Psi^{\prime\prime}(\Phi(-q_\ast))}
 \\
 &\qquad =  \lim_{\lambda\downarrow 0} \left(\frac{2y}{\lambda\Psi^{\prime\prime}(\Phi(-q_\ast))} - \frac{y^2}{\Psi^{\prime\prime}(\Phi(-q_\ast))} +o(1) \right)
\left(1 - \frac{2\lambda}{6}\frac{ \Psi^{\prime\prime\prime}(\Phi(-q_\ast) )}{ \Psi^{\prime\prime}(\Phi(-q_\ast) )}  +o(\lambda) \right) - \frac{2y}{\lambda \Psi^{\prime\prime}(\Phi(-q_\ast))}
\\
& \qquad =  - \frac{2 y \Psi^{\prime\prime\prime}(\Phi(-q_\ast))}{3 \left(\Psi^{\prime\prime}(\Phi(-q_\ast))\right)^2} - \frac{y^2}{ \Psi^{\prime\prime}(\Phi(-q_\ast))}.
\end{align*}
As a consequence, we deduce that
$$Q(y) =- \frac{2 y \Psi^{\prime\prime\prime}(\Phi(-q_\ast))}{3 \left(\Psi^{\prime\prime}(\Phi(-q_\ast))\right)^2} + \frac{y^2}{ \Psi^{\prime\prime}(\Phi(-q_\ast))} -  \int_0^{y}W_{\Phi(-q_\ast)}(a) da \leq y(q_1+q_2 y)
$$
for some positive constants $q_1,q_2$. Now, let $x\geq y$. Then, 
\begin{align*}
Q(y)&\geq  \int_y^{x}  \left(W_{\Phi(-q_\ast)}(a)-W_{\Phi(-q_\ast)}(a-y) - \frac{2y}{\Psi^{\prime\prime}(\Phi(-q_\ast))}\right) da\\
 &\geq  \left(W_{\Phi(-q_\ast)}(x)-W_{\Phi(-q_\ast)}(x-y) - \frac{2y}{\Psi^{\prime\prime}(\Phi(-q_\ast))}\right) (x-y)
\end{align*}
and we thus obtain the bound 
\begin{multline*}
x\left( W_{\Phi(-q_\ast)}(x)-W_{\Phi(-q_\ast)}(x-y) - \frac{2y}{\Psi^{\prime\prime}(\Phi(-q_\ast))}\right)  \\
 \leq  y \sup_{z\geq y}\left(W_{\Phi(-q_\ast)}(z)-W_{\Phi(-q_\ast)}(z-y)\right)   +y (q_1 + q_2 y)    .
\end{multline*}
But, from Point $(i)$,  
$$\sup_{z\geq y}\left(W_{\Phi(-q_\ast)}(z)-W_{\Phi(-q_\ast)}(z-y)\right) \leq  W_{\Phi(-q_\ast)}(y)$$
and Point $(ii)$ follows by combining this estimate with the asymptotics (\ref{eq:limWq}).
\end{proof}

\subsection{On the moments of ${\bf Z}_{0<x+1}$}

The general idea of the proof in the  critical case is similar to that of the subcritical case. But since here $\gamma=1$, we only need to study the first two moments of ${\bf Z}_{0<x+1}$.

\begin{lemma}\label{lem:EZq}
Fix $y> -1$. There exist two constants $\kappa_1,\kappa_2>0$, independent of $y$, such that for $x\geq y+1$ :
$$ \kappa_1 (1+y) \frac{e^{\Phi(-q_\ast)(x-y)}}{W_{\Phi(-q_\ast)}(x+1)}  \leq  \E_{x-y}\left[ {\bf Z}_{0<x+1}\right] \leq  \kappa_2 (1+ W_{\Phi(-q_\ast)}(1+y)) \frac{e^{\Phi(-q_\ast)(x-y)}}{W_{\Phi(-q_\ast)}(x+1)}.$$
\end{lemma}

\begin{proof}

Starting from Lemma \ref{lem:Mn} and applying the Esscher transform (\ref{eq:Esscher}), we have similarly to the subcritical case:
\begin{equation}\label{eq:EZ0xqL}
 \E_{x-y}\left[ {\bf Z}_{0<x+1}\right] = e^{\Phi(-q_\ast)(x-y)}  \E_{x-y}^{\Phi(-q_\ast)}\left[e^{-\Phi(-q_\ast)L_{\tau_0^-}}  1_{\{\tau_0^-<\tau_{x+1}^+ \}}    \right] . 
 \end{equation}
Since $L_{\tau_0^-}\leq 0$, the exponential term is always greater than one, hence from  \cite[Theorem 8.1]{Kyp}
$$\E_{x-y}\left[ {\bf Z}_{0<x+1}\right]  \geq e^{\Phi(-q_\ast)(x-y)}  \Pb_{x-y}^{\Phi(-q_\ast)}\left(\tau_0^-<\tau_{x+1}^+\right) = e^{\Phi(-q_\ast)(x-y)} \frac{W_{\Phi(-q_\ast)}(x+1) - W_{\Phi(-q_\ast)}(x-y) }{W_{\Phi(-q_\ast)}(x+1) }.$$
As a consequence,  we obtain from Point $(i)$ of Lemma \ref{lem:Wq} the lower bound: 
$$\E_{x-y}\left[ {\bf Z}_{0<x+1}\right]  \geq e^{\Phi(-q_\ast)(x-y)}  \frac{2(y+1)}{\Psi^{\prime\prime}(\Phi(-q_\ast))} \frac{1}{W_{\Phi(-q_\ast)}(x+1)}.$$
We now study the upper bound.  Starting from (\ref{eq:EZ0xqL}) and applying the Markov property together with \cite[Theorem 8.1]{Kyp}, we deduce from the absence of positive jumps that 
$$
\E_{x-y}\left[ {\bf Z}_{0<x+1}\right] =e^{\Phi(-q_\ast)(x-y)} \left(\E_{x-y}^{\Phi(-q_\ast)}\left[e^{-\Phi(-q_\ast)L_{\tau_0^-}}  \right]  - \frac{W_{\Phi(-q_\ast)}(x-y)}{W_{\Phi(-q_\ast)}(x+1)} \E_{x+1}^{\Phi(-q_\ast)}\left[e^{-\Phi(-q_\ast)L_{\tau_0^-}}  \right]\right).
$$
Let us set 
\begin{align} 
\label{eq:Hdef}H(x+1,x-y)&:=W_{\Phi(-q_\ast)}(x+1)  e^{-\Phi(-q_\ast)(x-y)} \E_{x-y}\left[ {\bf Z}_{0<x+1}\right]  \\
\notag&=   W_{\Phi(-q_\ast)}(x+1) \E_{x-y}^{\Phi(-q_\ast)}\left[e^{-\Phi(-q_\ast)L_{\tau_0^-}}  \right]  - W_{\Phi(-q_\ast)}(x-y) \E_{x+1}^{\Phi(-q_\ast)}\left[e^{-\Phi(-q_\ast)L_{\tau_0^-}}  \right].\end{align}
Our aim is thus to prove that this expression is bounded from above by a quantity independent of $x$. To do so, let us recall  from Kyprianou \cite[Corollary 10.2 and Exercise 10.1]{Kyp} that the distribution of $L_{\tau_0^-}$ is given on $(-\infty,0]$  by 
\begin{multline*}
\Pb_a(L_{\tau_0^-}\in dz) = \int_0^{+\infty} \left(e^{-\Phi_+(0) u} W(a) - W(a-u)\right) \nu(dz-u) du  \\+ \frac{\sigma^2}{2} 
\left(W^\prime(a) - \Phi_+(0) W(a)\right)\times  \delta_0(dz)  
\end{multline*}
where $\delta_0$ denotes the Dirac measure at 0.  Note that when $\sigma>0$, it is known, see \cite[Lemma 2.4]{KKR}, that $W$ is of class $\mathcal{C}^1$ on $(0,+\infty)$, hence the derivative $W^\prime$ is well-defined. Applying the Esscher transform,  we thus obtain from (\ref{eq:esscherparam}):
\begin{multline}\label{eq:EaePhi}
\E_{a}^{\Phi(-q_\ast)}\left[e^{-\Phi(-q_\ast)L_{\tau_0^-}}  \right] \\
=  \int_0^{+\infty} e^{-\Phi(-q_\ast) u} \left(W_{\Phi(-q_\ast)}(a) - W_{\Phi(-q_\ast)}(a-u)\right) \nu((-\infty,-u)) du  + \frac{\sigma^2}{2} 
W_{\Phi(-q_\ast)}^\prime(a)
\end{multline}
where we used that $\nu_{\Phi(-q_\ast)}(dx) = e^{\Phi(-q_\ast)x}\nu(dx)$ and that $\Psi_{\Phi(-q_\ast)}^\prime(0^+)=0$. 
Let us introduce the constant 
$$ \chi_{q_\ast}  = \frac{1}{\Psi^{\prime\prime}(\Phi(-q_\ast))}  \left(\int_0^{+\infty} 2 u  e^{-\Phi(-q_\ast)u} \nu((-\infty, -u)) du +   \sigma^2\right).$$
Using that $W_{\Phi(-q_\ast)}$ is increasing, we rewrite $H$ under the form
\begin{align}
\notag H(x+1, x-y) &= I(x+1, x-y) + W_{\Phi(-q_\ast)}(x-y)  \left( \Delta(x-y) - \Delta(x+1) \right)\\
\label{eq:H}&\leq I(x+1, x-y) + W_{\Phi(-q_\ast)}(x-y)  |\Delta(x-y) | +  W_{\Phi(-q_\ast)}(x+1)  | \Delta(x+1) |
\end{align}
where, 
$$  I(x+1, x-y) :=\left(W_{\Phi(-q_\ast)}(x+1) -  W_{\Phi(-q_\ast)}(x-y) \right)  \E_{x-y}^{\Phi(-q_\ast)}\left[e^{-\Phi(-q_\ast)L_{\tau_0^-}}  \right] $$
and
for $a\geq1$
$$\Delta(a):=\E_{a}^{\Phi(-q_\ast)}\left[e^{-\Phi(-q_\ast)L_{\tau_0^-}}  \right] - \chi_{q_\ast}.$$
Now, observe that by concavity, for $x\geq y+1$,
\begin{equation}\label{eq:boundq1}
  I(x+1, x-y) \leq W_{\Phi(-q_\ast)}(y+1) \sup_{a\geq 1} \E_{a}^{\Phi(-q_\ast)}\left[e^{-\Phi(-q_\ast)L_{\tau_0^-}}  \right]
\end{equation}
which is finite thanks to Lemma \ref{lem:C1} in the Appendix. Next, using the explicit expression (\ref{eq:EaePhi}) for the Laplace transform of $L_{\tau_0^-}$, we deduce that:
\begin{multline*}
\left|\Delta(a)\right| \leq 
\int_0^{+\infty} \left|W_{\Phi(-q_\ast) }(a) - W_{\Phi(-q_\ast) }(a-u)-\frac{2u }{\Psi^{\prime\prime}(\Phi(-q_\ast))}\right|e^{-\Phi(-q_\ast) u}  \nu((-\infty, -u)) du \\ + \frac{\sigma^2}{2} 
\left|W_{\Phi(-q_\ast)}^\prime(a)- \frac{2}{\Psi^{\prime\prime}(\Phi(-q_\ast))} \right|
\end{multline*}
and we decompose this expression in three parts. 
\begin{enumerate}
\item First, when $u\leq a$, we deduce from Point $(ii)$ of Lemma \ref{lem:Wq} that 
\begin{multline*}
 \int_0^{a} \left(W_{\Phi(-q_\ast) }(a) - W_{\Phi(-q_\ast) }(a-u)-\frac{2u }{\Psi^{\prime\prime}(\Phi(-q_\ast))}\right)e^{-\Phi(-q_\ast) u}  \nu((-\infty, -u)) du \\\leq  \frac{1}{a} \int_0^{+\infty} u (c_1+c_2 u) e^{-\Phi(-q_\ast) u}  \nu((-\infty, -u)) du.
 \end{multline*}
 Note that this last integral is finite since $\int_{-1}^{0}u^2 \nu(du)<+\infty$.
\item Then, when $u\geq a\geq 1$, the integral is smaller than
 \begin{multline*}
 \int_{a}^{+\infty} \left|W_{\Phi(-q_\ast) }(a) -\frac{2u}{\Psi^{\prime\prime}(\Phi(-q_\ast))}\right|e^{-\Phi(-q_\ast) u}  \nu((-\infty, -u)) du\\ \leq   e^{-\frac{\Phi(-q_\ast)}{2}(a)} \int_{1}^{+\infty} \left(W_{\Phi(-q_\ast) }(a) +\frac{2u }{\Psi^{\prime\prime}(\Phi(-q_\ast))}\right)e^{-\frac{\Phi(-q_\ast)}{2} u}  \nu((-\infty, -u)) du.
  \end{multline*}
\item Finally,  when $\sigma>0$, the concavity assumption implies that the function $W_{\Phi(-q_\ast)}^\prime$ is non-increasing and we deduce from (\ref{eq:limWq}) that
\begin{equation}\label{eq:limW'}
\lim_{x\rightarrow+\infty}W_{\Phi(-q_\ast)}^\prime(x) = \frac{2}{\Psi^{\prime\prime}(\Phi(-q_\ast))}.
\end{equation}
As a consequence, integrating by parts (\ref{eq:W}) and using the fact that $W_{\Phi(-q_\ast)}(0)=0$ since $\sigma>0$, see \cite[Lemma 3.1]{KKR}, we obtain from the monotone convergence theorem that
\begin{align*}
\int_0^{+\infty}  \left(W_{\Phi(-q_\ast)}^\prime(x) -  \frac{2}{\Psi^{\prime\prime}(\Phi(-q_\ast))}\right) dx  & = \lim_{\lambda\rightarrow 0} \int_0^{+\infty}  e^{-\lambda x}  \left(W_{\Phi(-q_\ast)}^\prime(x) -  \frac{2}{\Psi^{\prime\prime}(\Phi(-q_\ast))}\right) dx\\
&=\lim_{\lambda\rightarrow 0}  \frac{\lambda}{\Psi(\lambda+\Phi(-q_\ast)) +q_\ast} - \frac{2}{ \lambda \Psi^{\prime\prime}(\Phi(-q_\ast))}\\
&= - \frac{2\Psi^{\prime\prime\prime}(\Phi(-q_\ast))}{3 (\Psi^{\prime\prime}(\Phi(-q_\ast)))^2 }.
\end{align*}
In particular, since the integrand is non-negative
\begin{align*}
 - \frac{2\Psi^{\prime\prime\prime}(\Phi(-q_\ast))}{3 (\Psi^{\prime\prime}(\Phi(-q_\ast)))^2 }& \geq \int_0^{a}  \left(W_{\Phi(-q_\ast)}^\prime(x) -  \frac{2}{\Psi^{\prime\prime}(\Phi(-q_\ast))}\right) dx\\
 & \geq \left(W_{\Phi(-q_\ast)}^\prime(a) - \frac{2}{\Psi^{\prime\prime}(\Phi(-q_\ast))}\right) a.
\end{align*}
\end{enumerate}
Combining all these, we have thus obtained that there exists a constant $C>0$, independent of $a$, such that for $a\geq 1$ :
\begin{multline}\label{eq:WE}
 W_{\Phi(-q_\ast) }(a) \left|\E_{a}^{\Phi(-q_\ast)}\left[e^{-\Phi(-q_\ast)L_{\tau_0^-}}  \right] - \chi_{q_\ast}\right|\\
 \leq  C\left( \sup_{z\geq1} \frac{W_{\Phi(-q_\ast)}(z)}{z} +  \sup_{z\geq 1}   e^{-\frac{\Phi(-q_\ast)}{2}z}   \Big(W^2_{\Phi(-q_\ast)}(z)+W_{\Phi(-q_\ast)}(z)\Big)  \right)<+\infty
\end{multline}
which is finite thanks to (\ref{eq:limWq}). Plugging  (\ref{eq:boundq1}) and  (\ref{eq:WE}) in (\ref{eq:H}) finally yields the upper bound of Lemma \ref{lem:EZq}.
\end{proof}

\begin{lemma}\label{lem:Mq2}
There exists a constant $K>0$ independent of $a\geq 1$ such that for  $x>a-1$,
$$\E_a[{\bf Z}_{0<x+1}^2] \leq   K e^{\Phi(-q_\ast)a} \left( 1+ e^{\Phi(-q_\ast)(x+1)} \frac{W_{\Phi(-q_\ast)}(a)}{W^3_{\Phi(-q_\ast)}(x+1)}\right).$$
\end{lemma}

\begin{proof}
To study the second moment of ${\bf Z}_{0<x}$, we start again from Lemma \ref{lem:Mn}, using the bound (\ref{eq:boundRn}):
\begin{align*}
\E_a\left[ {\bf Z}_{0<x+1}^2\right]  &\leq  \E_a\left[ {\bf Z}_{0<x+1}\right]  +\E[\boldsymbol{p}^2] \beta  \int_0^{+\infty} \E_{a}\left[1_{\{\tau_0^-\wedge \tau_{x+1}^+ > r \}} \E_{L_r}\left[ {\bf Z}_{0<x+1}\right]^2\right] e^{q_\ast r}dr\\
&= \E_a\left[ {\bf Z}_{0<x+1}\right]  +\frac{\E[\boldsymbol{p}^2]  \beta }{W^2_{\Phi(-q_\ast)}(x+1)}  \int_0^{+\infty} \E_{a}\left[1_{\{\tau_0^-\wedge \tau_{x+1}^+> r \}} e^{2\Phi(-q_\ast)L_r}  H^2( x+1,L_r)\right] e^{q_\ast r}dr\\
&=\E_a\left[ {\bf Z}_{0<x+1}\right]  +\frac{ \E[\boldsymbol{p}^2] \beta e^{\Phi(-q_\ast)a} }{W^2_{\Phi(-q_\ast)}(x+1)} \int_0^{+\infty} \E_{a}^{\Phi(-q_\ast)}\left[1_{\{\tau_0^-\wedge \tau_{x+1}^+ > r \}} e^{\Phi(-q_\ast)L_r}  H^2(x+1 ,L_r)\right] dr
\end{align*}
where $H$ is given in (\ref{eq:Hdef}), $\beta=q_\ast/(m-1)$,  and we used the Esscher transform (\ref{eq:Esscher}) in the last equality.
From (\ref{eq:EZq}), the first term is bounded by 
$$\E_a\left[ {\bf Z}_{0<x+1}\right]  \leq e^{\Phi(-q_\ast) a} \sup_{z\geq 1} \E_z^{\Phi(-q_\ast)}\left[e^{-\Phi(-q_\ast) L_{\tau_0^-}}\right].$$
Then, using the formula (\ref{eq:resolvent}) for the resolvent of the killed process and a change of variable, we deduce that the integral is smaller than
\begin{align*}
&\int_0^{x+1} e^{\Phi(-q_\ast)y}  H^2(x+1, y) \frac{W_{\Phi(-q_\ast)}(a)W_{\Phi(-q_\ast)}(x+1-y)}{W_{\Phi(-q_\ast)}(x+1)} dy\\
&\qquad \qquad = e^{\Phi(-q_\ast)(x+1)} \frac{W_{\Phi(-q_\ast)}(a)}{W_{\Phi(-q_\ast)}(x+1)} \int_0^{x+1} e^{-\Phi(-q_\ast)z}  H^2(x+1, x+1-z) W_{\Phi(-q_\ast)}(z)dz .
\end{align*}
It remains thus to prove that this last integral is bounded as $x\rightarrow +\infty$. Using Lemma \ref{lem:EZq} with $y=z-1$, we have 
$$
 \int_0^{x} e^{-\Phi(-q_\ast)z}  H^2(x+1, x+1-z) W_{\Phi(-q_\ast)}(z)dz \leq \kappa_2^2 \int_0^{+\infty} e^{-\Phi(-q_\ast)z}  (1+W_{\Phi(-q_\ast)}(z))^2W_{\Phi(-q_\ast)}(z)dz$$
which is finite thanks to (\ref{eq:limWq}). Furthermore
\begin{multline*}
  \int_x^{x+1} e^{-\Phi(-q_\ast)z}  H^2(x+1, x+1-z) W_{\Phi(-q_\ast)}(z)dz\\ =  e^{-\Phi(-q_\ast)x} \int_0^{1} e^{-\Phi(-q_\ast)y}  H^2(x+1, 1-y) W_{\Phi(-q_\ast)}(y+x)dy.
  \end{multline*}
 By definition, the function $H$ reads, for $y\in(0,1)$,
\begin{align*} H(x+1, 1-y) &= W_{\Phi(-q_\ast)}(x+1) e^{-\Phi(-q_\ast) (1-y)} \E_{1-y}\left[ {\bf Z}_{0 < x+1}\right]\\
& \leq W_{\Phi(-q_\ast)}(x+1) \sup_{z\in(0,1)}\E_z\left[{\bf Z}_0 \right]
\end{align*}
hence, since $W_{\Phi(-q_\ast)}$ is increasing, 
$$
  \int_x^{x+1} e^{-\Phi(-q_\ast)z}  H^2(x+1, x+1-z) W_{\Phi(-q_\ast)}(z)dz \leq   e^{-\Phi(-q_\ast)x} W^3_{\Phi(-q_\ast)}(x+1)\sup_{z\in(0,1)}\E_z\left[{\bf Z}_0 \right]^2 $$
which is also bounded as $x\rightarrow+\infty$ thanks to  (\ref{eq:limWq}).
\end{proof}

We may now finish the proof of Theorem \ref{theo:1}, following the same steps as in the subcritical case.

\subsection{Upper bound}

Let us come back to (\ref{eq:UpB})
$$\Pb_a({\bf Z}_0 \geq n) \leq \Pb_a({\bf Z}_{0<x}\geq n) +  \Pb_a( {\bf M}\geq x)  \leq \frac{1}{n^2}\E_a[{\bf Z}_{0<x}^2] +  \Pb_a( {\bf M}\geq x)$$
and let us take a sequence  $(x_n)_{n\geq1}$ such that $\frac{e^{\Phi(-q_\ast) x_n}}{x_n} = n$. Recall that $a\geq1$. From  Lemma \ref{lem:Mq2} and the asymptotics of ${\bf M}$ given in (\ref{eq:AsymptMq}), we obtain that for $n$ large enough, 
\begin{align*}
\Pb_a({\bf Z}_0 \geq n)&\leq  K e^{\Phi(-q_\ast)a} \left( \frac{1}{n^2}+ e^{\Phi(-q_\ast)} \frac{1}{n x_n^2} \frac{W_{\Phi(-q_\ast)}(a) x_n^3 }{W^3_{\Phi(-q_\ast)}(x_n+1)}\right) +     \kappa_2 W^{(-q_\ast)}(a)  \frac{1}{n x_n^2}\\
& \leq  \frac{W^{(-q_\ast)}(a) }{n x_n^2} \left( \frac{ K x_n^2}{n W_{\Phi(-q_\ast)}(1)} +e^{\Phi(-q_\ast)} \frac{ K x_n^3 }{W^3_{\Phi(-q_\ast)}(x_n+1)} +   \kappa_2   \right).
\end{align*}
The upper bound now follows from (\ref{eq:limWq}) and from the fact that  $x_n \equi_{n\rightarrow+\infty} \frac{1}{\Phi(-q_\ast)} \ln(n)$.

\subsection{Lower bound}

Let us come back to (\ref{eq:LowB}): 
$$
 \Pb_a({\bf Z}_0 \geq n) \geq \Pb_x\left({\bf Z}_{0<x+1} \geq n \right) \Pb_a({\bf M}\geq x ).
$$
We take as before a sequence $(x_n)_{n\geq 1}$ such that $\E_{x_n}\left[ {\bf Z}_{0<x_n+1}\right]=2n$.  Lemma \ref{lem:EZq} applied with $y=0$ implies that as $n\rightarrow +\infty$, 
$$c_1\, n \leq   \frac{e^{\Phi(-q_\ast) x_n}}{x_n} \leq c_2 \,n $$  
for some constants $c_1, c_2>0$. Applying Paley-Zygmund's inequality as in the subcritical case, we deduce from (\ref{eq:AsymptMq}) that as $n\rightarrow +\infty$,
\begin{equation}\label{eq:lowZq}
 \Pb_a({\bf Z}_0 \geq n) \geq  \frac{1}{2} \frac{\kappa_1}{c_2}   \frac{n^2}{\E_{x_n}[{\bf Z}_{0<x_n+1}^2]} W^{(-q_\ast)}(a)\frac{1}{n x_n^2}
 \end{equation}
 for some constant $\kappa_1>0$ independent of $a>0$.
But, from Lemma \ref{lem:Mq2}, the second moment is bounded, for $n$ large enough, by 
\begin{align*}
\E_{x_n}[{\bf Z}_{0<x_n+1}^2] &\leq  K  e^{\Phi(-q_\ast)x_n} +  Ke^{\Phi(-q_\ast)(2 x_n + 1 )} \frac{W_{\Phi(-q_\ast)}(x_n)}{W^3_{\Phi(-q_\ast)}(x_n+1)}\\
&\leq   K c_2 \left( n x_n +   e^{\Phi(-q_\ast) }   n^2  \frac{x_n^2 W_{\Phi(-q_\ast)}(x_n)}{W^3_{\Phi(-q_\ast)}(x_n+1)}\right)\leq  A  n^2  \left( 1 +    \sup_{z\geq1} \frac{z^2}{W^2_{\Phi(-q_\ast)}(z)}\right)
\end{align*}
for some constant $A>0$  since as before  $x_n \equi_{n\rightarrow+\infty} \frac{1}{\Phi(-q_\ast)} \ln(n)$. The lower bound now follows by plugging this last inequality in (\ref{eq:lowZq}).
\qed

\section{Appendix: a bound on the undershoot}\label{sec:app}

We give below a uniform control on some exponential moments of the undershoot of $L$ below 0. The convergence in law of the undershoot has already been studied in Roynette, Vallois \& Volpi \cite{RVV}, but we need to work here with unbounded functions, which requires some extra care. This is the purpose of the following lemma, whose proof relies on an argument somewhat different from the rest of the paper.

\begin{lemma}\label{lem:C1}
Assume that $\beta(m-1)< q_\ast$, or $\beta(m-1)= q_\ast$ and  the function $x\rightarrow e^{-\Phi(-q_\ast) x} W^{(-q_\ast)}(x)$ is concave. Then, it holds
\begin{equation}\label{eq:SupL0}
\sup_{a\geq0} \E_{a}^{\Phi_-(\beta(1-m))}\left[e^{-\Phi_-(\beta(1-m))L_{\tau_0^-}}  \right] <+\infty.
 \end{equation}
Moreover, when $\beta(m-1)=q_\ast$, we have the limit
\begin{equation}\label{eq:chiq}
\chi_{q_\ast}:= \lim_{a\rightarrow+\infty} \E_{a}^{\Phi(-q_\ast)}\left[e^{-\Phi(-q_\ast)L_{\tau_0^-}}  \right]  = \frac{1}{\Psi^{\prime\prime}(\Phi(-q_\ast))}  \left(\int_0^{+\infty} 2 u  e^{-\Phi(-q_\ast)u} \nu((-\infty, -u)) du +   \sigma^2\right) .
 \end{equation}
\end{lemma}

\begin{proof}
We start with the subcritical case $\beta(m-1)< q_\ast$.  Observe that we may restrict our attention to the case $a\geq1$ thanks to (\ref{eq:Etau0}).
We first separate the small jumps from the big jumps:
\begin{equation}\label{eq:creeping}
 \E_{a}^{\Phi_-(\beta(1-m))}\left[e^{-\Phi_-(\beta(1-m))L_{\tau_0^-}}  \right] \leq  \E_{a}^{\Phi_-(\beta(1-m))}\left[e^{-\Phi_-(\beta(1-m))L_{\tau_0^-}}  1_{\{L_{\tau_0^-}\leq -1\}} \right] + e^{\Phi_-(\beta(1-m))}. 
 \end{equation}
Now $\Psi_{\Phi_-(\beta(1-m))}^\prime(0^+)<0$, so that the L\'evy process $L$ goes a.s. towards $-\infty$ under  $\Pb_{a}^{\Phi_-(\beta(1-m))}$.  
As a consequence, instead of looking at the first jump below 0 of magnitude larger than 1, we shall rather look at all the jumps starting above 0 and going below $-1$.
The compensation formula (see \cite[Chapter O.5, p.7]{Ber}) and (\ref{eq:esscherparam}) then yield: 
\begin{align}
\notag &\E_{a}^{\Phi_-(\beta(1-m))}\left[e^{-\Phi_-(\beta(1-m))L_{\tau_0^-}} 1_{\{L_{\tau_0^-}\leq -1\}}   \right]\\
\notag  &\qquad \leq  \E_{a}^{\Phi_-(\beta(1-m))}\left[ \sum_{t \geq0}   e^{-\Phi_-(\beta(1-m))(L_{t-} + \Delta L_t)} 1_{\{L_{t-} >0;\; L_{t-} + \Delta L_t<-1  \}} \right]  \\
\notag &\qquad=\E_{a}^{\Phi_-(\beta(1-m))}\left[ \int_0^{+\infty} dt \int_{-\infty}^0 e^{-\Phi_-(\beta(1-m))(L_{t} + s)} 1_{\{L_{t} >0;\; L_{t} + s<-1  \}}   \nu_{\Phi_-(\beta(1-m))}(ds) \right]  \\
\notag &\qquad=\E_{a}^{\Phi_-(\beta(1-m))}\left[ \int_0^{+\infty} e^{-\Phi_-(\beta(1-m))L_{t}} 1_{\{L_t>0\}}   \nu((-\infty, \,-L_t-1 ))  dt \right]  \\
\label{eq:boundmaster} &\qquad\leq   \nu((-\infty, \,-1 ))  \E_{a}^{\Phi_-(\beta(1-m))}\left[ \int_0^{+\infty}  e^{-\Phi_-(\beta(1-m))L_{t}} 1_{\{L_t>0\}} dt \right] . 
\end{align}
We now define for $x\in \R$ the renewal function $U$ associated to $L$ by: 
$$U(x) =   \int_0^{+\infty}   \Pb_{0}^{\Phi_-(\beta(1-m))}( L_t \geq  x) dt$$
which is well-defined since $\beta(m-1)< q_\ast$. Indeed, applying the Fubini-Tonelli theorem for $\lambda>0$ small enough such that $\Psi_{\Phi_-(\beta(1-m))}(\lambda)<0$, it holds:
\begin{multline*}
\int_\R  e^{\lambda x}  U(x) dx =\frac{1}{\lambda} \int_0^{+\infty}  \E_0^{\Phi_-(\beta(1-m))}\left[e^{\lambda L_t}\right]dt  \\=      \frac{1}{\lambda} \int_0^{+\infty} e^{t \Psi_{\Phi_-(\beta(1-m))}(\lambda)} dt = - \frac{1}{\lambda \Psi_{\Phi_-(\beta(1-m))}(\lambda)} <+\infty.
\end{multline*}
Note also that  $U$ is subadditive, i.e.
\begin{equation}\label{eq:subad}
U(x+y) \leq U(x) + U(y),\qquad \forall x,y \in \R.
\end{equation}
Indeed, on the one hand, taking $x,y\geq0$ and applying the Markov property,
\begin{align*}
U(-x-y) &= \E_0^{\Phi_-(\beta(1-m))}\left[ \int_0^{\tau_{-y}^-}  1_{\{L_t\geq -x-y\}} dt   \right] +  \E_0^{\Phi_-(\beta(1-m))}\left[ \int_{\tau_{-y}^-}^{+\infty}  1_{\{L_t\geq -x-y\}} dt   \right] \\
&\leq \E_0^{\Phi_-(\beta(1-m))}\left[ \int_0^{\tau_{-y}^-}  1_{\{L_t\geq -y\}} dt   \right] +  \E_0^{\Phi_-(\beta(1-m))}\left[ \int_{0}^{+\infty}  1_{\left\{L_{t+\tau_{-y}^-}\geq -x-y\right\}} dt   \right] \\
&\leq U(-y) +  \E_0^{\Phi_-(\beta(1-m))}\left[ \int_{0}^{+\infty}  1_{\left\{ \widehat{L}_{t}\geq -x-y-L_{\tau_{-y}^{-}}\right\}} dt   \right] \leq U(-y) + U(-x)
\end{align*}
where  $ \widehat{L} $ is an independent copy of $L$, and where we used in the second line the fact that $L_t\geq -y$ on $(0, \tau_{-y}^-)$, and on the last inequality that $-y-L_{\tau_{-y}^{-}}\geq0$.
On the other hand, taking $x\geq0$ and $y\in \R$, and using that $U$ is positive and decreasing,
$$U(x) + U(y) \geq U(x) + U(x+y) \geq U(x+y).$$
Now going back to (\ref{eq:boundmaster}) and decomposing the spatial integral, we deduce that
\begin{align*}
&\E_{a}^{\Phi_-(\beta(1-m))}\left[ \int_0^{+\infty}  e^{-\Phi_-(\beta(1-m))L_{t}} 1_{\{L_t>0\}} dt \right]  \\
 &\qquad \leq    \sum_{k=0}^{+\infty}  e^{-\Phi_-(\beta(1-m))k} \int_k^{k+1}   \int_0^{+\infty}   \Pb_{a}^{\Phi_-(\beta(1-m))}( L_t \in dx) dt\\
&\qquad \leq  \sum_{k=0}^{+\infty}  e^{-\Phi_-(\beta(1-m))k} \left(U(k-a) - U(k+1-a)\right)\\
&\qquad \leq  U(-1) \sum_{k=0}^{+\infty}  e^{-\Phi_-(\beta(1-m))k} <+\infty
\end{align*}
 which is finite and independent of $a$. \\
 
 We now examine the critical case $\beta(m-1)=q_\ast$. Note that the previous argument no longer works here, as $U$ is not well-defined in this case. We shall thus proceed differently by observing that it is sufficient to prove the existence of the limit (\ref{eq:chiq}).
 Let us start from the explicit expression  for the undershoot given in (\ref{eq:EaePhi}):
 \begin{multline*}
\E_{a}^{\Phi(-q_\ast)}\left[e^{-\Phi(-q_\ast)L_{\tau_0^-}}  \right] \\
=  \int_0^{+\infty} e^{-\Phi(-q_\ast) u} \left(W_{\Phi(-q_\ast)}(a) - W_{\Phi(-q_\ast)}(a-u)\right) \nu((-\infty,-u)) du  + \frac{\sigma^2}{2} 
W_{\Phi(-q_\ast)}^\prime(a).
\end{multline*}
The convergence of the creeping term is a direct consequence of (\ref{eq:limW'}). Let us then decompose the integral in the right-hand side in two terms. First, when $u\geq a$, the term $W_{\Phi(-q_\ast)}(a-u)$ is null and
\begin{align*}
W_{\Phi(-q_\ast)}(a)   \int_a^{+\infty} e^{-\Phi(-q_\ast) u} \nu((-\infty,-u)) du &\leq W_{\Phi(-q_\ast)}(a)  e^{-\frac{\Phi(-q_\ast)}{2} a}  \int_a^{+\infty} e^{-\frac{\Phi(-q_\ast)}{2} u} \nu((-\infty,-u)) du\\
&\xrightarrow[a\rightarrow+\infty]{}0.
\end{align*}
Next, we deduce from Lemma \ref{lem:Wq} that for $a\geq u$ and $a\geq 1$:
$$0\leq W_{\Phi(-q_\ast)}(a)-W_{\Phi(-q_\ast)}(a-u) \leq \frac{2 u}{\Psi^{\prime\prime}(\Phi(-q_\ast))} + u (c_1+c_2u).$$
As a consequence, we can apply the dominated convergence theorem, using the limit given in Point $(i)$ of Lemma \ref{lem:Wq} to obtain:
\begin{multline*}
\lim_{a\rightarrow+\infty} \int_0^{a} e^{-\Phi(-q_\ast) u} \left(W_{\Phi(-q_\ast)}(a) - W_{\Phi(-q_\ast)}(a-u)\right) \nu((-\infty,-u)) du\\ =  \frac{2}{\Psi^{\prime\prime}(\Phi(-q_\ast))}  \int_0^{+\infty}  u  e^{-\Phi(-q_\ast)u} \nu((-\infty, -u)) du.
\end{multline*}
This completes the proof of (\ref{eq:chiq}).
\end{proof}

\noindent
\textbf{Acknowledgments.} 
We are grateful to the referees for their careful reading and many helpful suggestions, as well as for pointing out  the missing argument given in the Appendix.

\end{document}